\documentclass[onefignum,onetabnum]{siamonline220329}
\usepackage{graphicx,appendix,comment}
\usepackage{booktabs}
\usepackage[english]{babel}
\usepackage[utf8]{inputenc}
\usepackage{amssymb}
\usepackage{multicol}
\usepackage[utf8]{inputenc}
\usepackage[english]{babel}
\usepackage{multirow}
\usepackage{url}
\usepackage{cleveref}
\usepackage{cancel} 
\usepackage{bbm}
\usepackage{xcolor}
\usepackage{algorithm}
\usepackage{algpseudocode}
\usepackage{bm}
\algrenewcommand\algorithmicrequire{\textbf{Input:}}
\algrenewcommand\algorithmicensure{\textbf{Output:}}
\usepackage[center]{subfigure}

\newcommand{\norm}[1]{\left\lVert#1\right\rVert}
\newcommand{\bfA}{{\bf A}}
\newcommand{\bfB}{{\bf B}}

\newcommand{\bfC}{{\bf C}}
\newcommand{\Cb}{{\bfC}_\bfb}

\newcommand{\bfD}{{\bf D}}
\newcommand{\bfE}{{\bf E}}

\newcommand{\bfI}{{\bf I}}

\newcommand{\bfK}{{\bf K}}
\newcommand{\bfL}{{\bf L}}
\newcommand{\bfM}{{\bf M}}
\newcommand{\tbfM}{\tilde{\bf M}}

\newcommand{\bfR}{{\bf R}}

\newcommand{\bfT}{{\bf T}}
\newcommand{\bfU}{{\bf U}}

\newcommand{\bfV}{{\bf V}}

\newcommand{\bfW}{{\bf W}}

\newcommand{\bfX}{{\bf X}}

\newcommand{\bfZ}{{\bf Z}}
\newcommand{\bfa}{{\bf a}}
\newcommand{\bfb}{{\bf b}}

\newcommand{\bfd}{{\bf d}}
\newcommand{\bfe}{{\bf e}}

\newcommand{\bfg}{{\bf g}}
\newcommand{\bfh}{{\bf h}}

\newcommand{\bfk}{{\bf k}}

\newcommand{\bfq}{{\bf q}}
\newcommand{\bfr}{{\bf r}}
\newcommand{\bfs}{{\bf s}}

\newcommand{\bfu}{{\bf u}}
\newcommand{\bfv}{{\bf v}}
\newcommand{\bfw}{{\bf w}}
\newcommand{\bfx}{{\bf x}}

\newcommand{\bfz}{{\bf z}}
\newcommand{\bfzero}{{\bf0}}

\newcommand{\bfzeta}{{\boldsymbol{\zeta}}}
\newcommand{\bfUpsilon}{{\boldsymbol{\Upsilon}}}
\DeclareMathOperator{\diag}{diag}
\newcommand{\tbfUpsilon}{\tilde{\boldsymbol{\Upsilon}}}

\newcommand{\bfPsi}{{\boldsymbol{\Psi}}}

\DeclareMathOperator*{\argmin}{arg\,min}
\newcommand*\samethanks[1][\value{footnote}]{\footnotemark[#1]}

\title{Parameter Selection by GCV and a $\chi^2$ test within Iterative Methods for $\ell_1$-regularized Inverse Problems
}

\author{Brian Sweeney\thanks{School of Mathematical and Statistical Sciences, Arizona State University, Tempe, AZ (\href{mailto:bfsweene@asu.edu}{bfsweene@asu.edu}, \href{mailto:renaut@asu.edu}{renaut@asu.edu},  and \href{mailto:malena.espanol@asu.edu}{malena.espanol@asu.edu} )} \and Rosemary Renaut\samethanks \and Malena I. Espa\~nol\samethanks} 
\headers{Parameter Estimation in SB and MM Methods}{Sweeney, Renaut and Espa\~nol}
\begin{document} 

\maketitle

\begin{abstract}
    $\ell_1$ regularization is used to preserve edges or enforce sparsity in a solution to an inverse problem. We investigate the Split Bregman and the Majorization-Minimization iterative methods that turn this non-smooth minimization problem into a sequence of steps that include solving an $\ell_2$-regularized minimization problem. We consider selecting the regularization parameter in the inner generalized Tikhonov regularization problems that occur at each iteration in these $\ell_1$ iterative methods. The generalized cross validation  and $\chi^2$ degrees of freedom  methods are extended to these inner problems. In particular, for the $\chi^2$ method this includes extending the $\chi^2$ result for problems in which the regularization operator has more rows than columns and showing how to use the $A-$weighted generalized inverse to estimate prior information at each inner iteration. Numerical experiments for image deblurring problems demonstrate that it is more effective to select the regularization parameter automatically within the iterative schemes than to keep it fixed for all iterations. Moreover, an appropriate regularization parameter can be estimated in the early iterations and used fixed to convergence.
\end{abstract}



\section{Introduction}
We are interested in solving discrete ill-posed inverse problems where we have an observation $\tilde{\bfb}$ from an unknown input $\bfx$, connected by the linear system $\tilde{\bfA} \bfx \approx \tilde{\bfb}$, where $\tilde{\bfA} \in \mathbb{R}^{m\times n}$, $\tilde{\bfb} \in \mathbb{R}^m$, and $\bfx \in \mathbb{R}^n$.  We assume that $\tilde{\bfA}$ is ill-conditioned and that $\tilde{\bfb}$ is contaminated by additive Gaussian noise: $\tilde{\bfb} = \tilde{\bfb}_{true} + \boldsymbol{\epsilon}_{\tilde{\bfb}}$, where $\boldsymbol{\epsilon}_{\tilde{\bfb}} \sim \mathcal{N}({\bfzero},\bfC_{\bfb})$ is a Gaussian noise vector.  The matrix $\bfC_{\bfb}$ is a symmetric positive definite (SPD) covariance matrix. Since $\tilde{\bfA}$ is ill-conditioned, solving the problem with direct inversion will lead to a noisy solution.  Therefore, we impose regularization to make the problem well-posed. One option is to apply Tikhonov regularization~\cite{tikhonov1963solution} and solve the minimization problem
\begin{align}\label{eq:GenTikO}
   \min_{\bfx}\left\{ \frac{1}{2}\|\tilde{\bfA} \bfx - \tilde{\bfb}\|_{\bfW_{\bfb}}^2 + \frac{\lambda^2}{2} \norm{\bfL \bfx}_2^2\right\},
\end{align}
where $\bfW_{\bfb} = \bfC_{\bfb}^{-1}$ and the weighted norm is defined as $\norm{\bfz}_{\bfW_\bfb}^2 = \bfz^\top \bfW_\bfb \bfz$ for any vector $\bfz$. 
Here, $\lambda$ is a regularization parameter that balances the data fidelity term and the regularization term. The regularization matrix $\bfL \in \mathbb{R}^{p\times n}$ is often selected as the discretization of a derivative operator \cite{hansen2010discrete}, and the regularization term then minimizes the corresponding derivative of $\bfx$. Other matrices, such as discrete wavelet \cite{fang2014wavelet} or framelet transforms \cite{cai2009linearized,wang2014ultrasound}, can be used to minimize the value of $\bfx$ in the corresponding subspaces. 

We can whiten the noise in the data  by multiplying out the weighted norm in \cref{eq:GenTikO}, giving the minimization problem 
\begin{align}\label{eq:GenTik}
    \min_{\bfx}\left\{ \frac{1}{2}\norm{\bfA \bfx - \bfb}_2^2 + \frac{\lambda^2}{2} \norm{\bfL \bfx}_2^2\right\},
\end{align}
where
\begin{align}
    \bfA = \bfW_\bfb^{1/2}\tilde{\bfA} \text{ and } \bfb = \bfW_\bfb^{1/2}\tilde{\bfb}. \label{eq:Reweighting}
\end{align}
Now, $\bfb = \bfb_{true} + \boldsymbol{\epsilon}_{\bfb}$, where $\boldsymbol{\epsilon}_{\bfb} \sim \mathcal{N}({\bfzero},\bfI_m)$.  For the rest of the paper, we use the weighted $\bfA$ and $\bfb$, meaning that we assume $\boldsymbol{\epsilon}_{\bfb} \sim \mathcal{N}({\bfzero},\bfI_m)$. 

One benefit of the $\ell_2$ regularization problem \cref{eq:GenTik} is that it has a closed-form solution, but it also makes solutions smooth.  If the true solution has edges or is sparse, smooth solutions may not be desirable. For these types of solutions, the 1-norm is typically used as it preserves edges and enforces sparsity.  This gives the $\ell_1$-regularized problem
\begin{align}
    \min_{\bfx} \left\{ \frac{1}{2}\norm{\bfA \bfx - \bfb}_2^2 + \mu\norm{\bfL \bfx}_1\right\}. \label{eq:TV1}
\end{align}
Here, $\mu$ is a regularization parameter and $\bfL$ is a regularization matrix as in \cref{eq:GenTik}.
One special case of $\ell_1$ regularization is total variation (TV) regularization, first introduced in \cite{rudin1992nonlinear}, where $\bfL$ is the discretization of the first derivative. Unlike \cref{eq:GenTik}, \cref{eq:TV1} does not have a closed-form solution.  {Matrices $\bfA$ and }$\bfL$ may not have full rank, but we assume the invertibility condition
\begin{align}
    \mathcal{N}(\bfA)\cap\mathcal{N}(\bfL) = \emptyset.
    \label{eq:invcond}
\end{align}

In both \cref{eq:GenTik} and \cref{eq:TV1}, the value of the regularization parameter has a large impact on the solution.  For \cref{eq:GenTik}, there are many methods for selecting $\lambda$, including the discrepancy principle (DP) 
\cite{hansen1998rank,morozov1966solution}, the unbiased predictive risk estimator (UPRE) \cite{mallows2000some}, generalized cross validation (GCV) \cite{golub1979generalized}, the L-curve criterion \cite{hansen1992analysis}, and the $\chi^2$ degrees of freedom (dof)  test \cite{mead2008newton}.
If training data sets are available, there are also learning approaches that can be used to select $\lambda$ \cite{byrne2023learning,chung2017learning}.
  Many of these methods make use of the closed-form solution, which makes applying them to \cref{eq:TV1} difficult. 
  As a result, there are fewer parameter selection methods for \cref{eq:TV1}. 
Some methods have been extended to selecting $\mu$ directly in \cref{eq:TV1}, including DP \cite{bonesky2008morozov,jin2012iterative} and the L-curve~\cite{hou2018selection, yao2011compressive}.  With the L-curve, the two terms in~\cref{eq:TV1} are plotted against each other on a log-scale, and $\mu$ is selected at the corner of the corresponding curve.  There are also methods in the statistics community for selecting $\mu$ based upon the degrees of freedom in the solution \cite{tibshirani2012degrees}, including the $\chi^2$ dof test for the TV problem~\cite{mead2020chi}. 
 Some methods have also been applied within iterative methods for \cref{eq:TV1}.  DP in its iterative form \cite{buccini2022comparison,buccini2019l}, GCV \cite{buccini2022comparison}, and the residual whiteness principle (RWP) \cite{buccini2022comparison,etna_vol53_pp329-351} have been applied to select parameters at each iteration within iterative methods. 

\noindent{{\bf Main contributions}}: We consider the Split Bregman (SB) and Majorization-Minimization (MM)   iterative methods for solving the $\ell_1$ regularization problem. 
Both iterative methods solve an $\ell_2$-$\ell_2$ minimization problem of the form
\begin{align}\label{eq:genTik_hk}
       \min_{\bfx} J(\bfx)= \min_{\bfx}\left\{ \frac{1}{2}\norm{\bfA \bfx - \bfb}_2^2 + \frac{\lambda^2}{2} \|\bfL \bfx - \bfh^{(k)}\|_2^2\right\},
\end{align}
at the $k^{\text{th}}$ iteration.
Based on this minimization problem, we consider  new methods for selecting  $\lambda$ at every iteration. In particular, we extend GCV and the $\chi^2$ dof test to this inner problem.   For GCV, we derive the GCV function for \cref{eq:genTik_hk}, which is different than it is for \cref{eq:GenTik} due to $\bfh^{(k)}$. To apply the $\chi^2$ dof test, we use the $\bfA$-weighted generalized inverse of $\bfL$ to replace $\bfh^{(k)}$ in \cref{eq:genTik_hk} so that we can apply the $\chi^2$ test with the regularization term $\|\bfL(\bfx-\bfx_0)\|_2$ for a suitably defined reference vector $\bfx_0$. We also extend the non-central $\chi^2$ test for this configuration, and provide a new result on the degrees of freedom for problems in which $\bfL$ has more rows than columns, as needed for $2$D cases. Through numerical examples, we show that these selection methods can be applied at each iteration to achieve results that are comparable to finding a fixed $\lambda$ that is optimal with respect to the minimization of the relative error in the solution.  We also demonstrate that GCV and the $\chi^2$ dof test can be used in the initial iterations to find a suitable regularization parameter.  The methods zoom in on a parameter that is then held fixed when the change in the parameter per iteration is less than a given tolerance.  This works well and is less expensive than estimating $\lambda$ by GCV or the $\chi^2$ estimator at every iteration.

The organization of this paper is as follows. In \cref{sec:iter-methods}, we review SB and MM iterative methods. In \cref{sec:reg}, we develop three  new methods to find the iterative dependent regularization parameter for the Tikhonov problem that arises in both the SB and MM algorithm, focusing on the GCV in \cref{subsec:GCV},  the $\chi^2$ dof test in \cref{subsec:chi2}, and the non-central $\chi^2$ dof test in \cref{subsubsec:noncentralchi}. In \cref{subsec:othermethods}, we present the DP and RWP, which are other established methods for selecting the parameter at each iteration.  These methods are  compared with numerical examples in \cref{sec:Examples}. Conclusions are presented in \cref{sec:Conclusion}.

\section{Iterative methods for \texorpdfstring{$\ell_1$}{l1} regularization}\label{sec:iter-methods}
In this section, we review the split Bregman (SB) and the Majorization-Minimization (MM) methods, which share the inner problem of the  form \cref{eq:genTik_hk}. Both methods are applied to the weighted $\bfA$ and $\bfb$ as defined in \cref{eq:Reweighting}.

\subsection{The Split Bregman method}\label{subsec:SB}
In the SB method, introduced by Goldstein and Osher \cite{GO}, the problem~\cref{eq:TV1} is rewritten as a constrained optimization problem
\begin{align}\label{eq:l11}
    \min_{\bfx}\left\{\frac{1}{2}\|\bfA \bfx - \bfb \|^2_2 + \mu \|\bfd \|_1\right\} \quad \text{ s.t. } \bfL \bfx = \bfd.
\end{align}
Problem \eqref{eq:l11} can then be converted to an unconstrained optimization problem
\begin{align} \label{eq:Breg1}
    \min_{\bfx,\bfd}\left\{\frac{1}{2}\|\bfA \bfx - \bfb \|^2_2 + \frac{\lambda^2}{2} \|\bfL \bfx - \bfd \|^2_2 + \mu \|\bfd \|_1\right\},
\end{align}
which can be solved by a series of minimizations and updates known as the SB iteration
\begin{align}\label{eq:Breg2}
    (\bfx^{(k+1)}, \bfd^{(k+1)}) &= \argmin_{\bfx,\bfd}\left\{\frac{1}{2}\|\bfA \bfx - \bfb \|^2_2 + \frac{\lambda^2}{2}\|\bfL \bfx - \bfd^{(k)}+ \bfg^{(k)} \|^2_2 + \mu \|\bfd \|_1\right\} \\
    \bfg^{(k+1)} &= \bfg^{(k)} + (\bfL \bfx^{(k+1)} - \bfd^{(k+1)}). \label{eq:Breg2a}
\end{align}
In Problem \eqref{eq:Breg2}  the vectors $\bfx$ and $\bfd$ can be found separately as
\begin{align}\label{eq:Breg3}
    \bfx^{(k+1)} &= \argmin_\bfx \left\{\frac{1}{2}\|\bfA \bfx-\bfb\|_2^2 +\frac{\lambda^2}{2} \|\bfL\bfx -(\bfd^{(k)}-\bfg^{(k)})\|_2^2\right\}  \\ 
\bfd^{(k+1)}&= \argmin_\bfd \left\{\mu \|\bfd \|_1+\frac{\lambda^2}{2} \|\bfd-(\bfL\bfx^{(k+1)}+\bfg^{(k)})\|_2^2 \right\}. \label{eq:Breg4} 
\end{align}
Here, and in the update \cref{eq:Breg2a}, $\bfg$ is the vector of Lagrange multipliers. 

Clearly, the solution of \cref{eq:Breg2} depends on parameters $\lambda$ and $\mu$ that are often chosen as problem-dependent known values.  In this investigation, we  hold the ratio $\tau=\mu^{(k)}/(\lambda^{(k)})^2$ fixed and explore methods to select $\lambda^{(k)}$ at each iteration.  This keeps the threshold constant across iterations whereas if we tune $\mu$ instead, the threshold would change at each iteration. With these parameters, $\mu^{(k)} = \tau(\lambda^{(k)})^2 $ and \cref{eq:Breg4} becomes
\begin{align}
    \bfd^{(k+1)}= \argmin_\bfd \left\{\tau \|\bfd \|_1+\frac{1}{2} \|\bfd-(\bfL\bfx^{(k+1)}+\bfg^{(k)})\|_2^2 \right\}. \label{eq:Breg5}
\end{align}
Since the elements of $\bfd$ are decoupled in \cref{eq:Breg5}, $\bfd^{(k+1)}$ can be computed using shrinkage operators. That is, each element is given by
\begin{align}
    d_j^{(k+1)} = \text{shrink}\left((\bfL\bfx^{(k+1)})_j+g^{(k)}_j,\tau\right),  \nonumber
\end{align}
where $\text{shrink}(x,\tau) = \text{sign}(x) \cdot \text{max}(|x|-\tau,0)$. \Cref{alg:SB} summarizes the SB algorithm.
Notice that SB is related to applying the alternating direction method of multipliers (ADMM) to the augmented Lagrangian in \cref{eq:Breg1} \cite{esser2009applications}.  

\begin{algorithm}
\caption{The SB Method for the $\ell_2$-$\ell_1$ Problem \cref{eq:TV1} }
\label{alg:SB}
\begin{algorithmic}[1]
\Require $\bfA, \bfb, \bfL,\tau, \bfd^{(0)} = \bfg^{(0)} = {\bfzero}$
\Ensure $\bfx$
    \For{$k=0,1,\dots$ until convergence}
    \State Estimate $\lambda^{(k)}$
    \State $\bfx^{(k+1)} = \argmin_\bfx \left\{\frac{1}{2}\|\bfA\bfx-\bfb\|_2^2 + \frac{\left(\lambda^{(k)}\right)^2}{2}\|\bfL\bfx-(\bfd^{(k)}-\bfg^{(k)})\|_2^2\right\}$ \label{line:xup}
    \State $\bfd^{(k+1)} = \argmin_\bfd \left\{ \tau\|\bfd\|_1 + \frac{1}{2}\|\bfd-(\bfL\bfx^{(k+1)}+\bfg^{(k)})\|_2^2\right\}$ \label{line:dup}
    \State $\bfg^{(k+1)} = \bfg^{(k)} + \left(\bfL\bfx^{(k+1)} - \bfd^{(k+1)}\right)$ \label{line:gup}
    \EndFor
\end{algorithmic}
\end{algorithm}

\subsection{The Majorization-Minimization method}\label{subsec:MM}
Another iterative method for solving \cref{eq:TV1} is the MM method.  MM is an optimization method that utilizes two steps: majorization and minimization \cite{hunter2004tutorial}.  In the majorization step, the function is majorized with a surrogate convex function.  The convexity of this function is then utilized in the minimization step. 
MM is applied to \cref{eq:TV1} in \cite{huang2017majorization}, where it is combined with the generalized Krylov subspace (GKS) method to solve large-scale image restoration problems.  In MM-GKS, the Krylov subspace is enlarged at each iteration, and MM is then applied to the problem in the subspace.  
Instead of building a Krylov subspace as in \cite{buccini2020lp,huang2017majorization, pasha2020krylov}, we will apply the MM method directly to \cref{eq:TV1}, using the fixed quadratic majorant from \cite{huang2017majorization}.  Other ways for majorizing \cref{eq:TV1} include fixed and adaptive quadratic majorants \cite{alotaibi2021restoration,buccini2020lp,pasha2020krylov}.  

With this majorant, the minimization problem at each $k^{\text{th}}$ iteration is
\begin{align} \label{eq:MMsolve}
    \min_{\bfx}\left\{\frac{1}{2}\|\bfA\bfx-\bfb\|^2_2 + \frac{\lambda^2}{2}\|\bfL\bfx-\bfw_{reg}^{(k)}\|^2_2\right\},
\end{align}
where $\lambda = (\mu/\varepsilon)^{1/2}$ and
\begin{align}
    \bfw_{reg}^{(k)} = \bfu^{(k)}\left(1-\left(\frac{\varepsilon^2}{(\bfu^{(k)})^2+\varepsilon^2}\right)^{\frac{1}{2}}\right) \label{eq:wreg}
\end{align}
with $\bfu^{(k)} = \bfL \bfx^{(k)}$. All operations in \cref{eq:wreg} are component-wise. Again, the parameters can be selected at each iteration. Here, as with our approach for the SB algorithm, we select $\lambda^{(k)}$ at each iteration and fix $\varepsilon$ as in \cite{buccini2022comparison} for MM-GKS. \cref{alg:MML} summarizes the MM algorithm.

\begin{algorithm}
\caption{The MM Method for the $\ell_2$-$\ell_1$ Problem \cref{eq:TV1} with a Fixed Quadratic Majorant}
\label{alg:MML} 
\begin{algorithmic}[1]
\Require $\bfA, \bfb, \bfL, \bfx^{(0)}, \varepsilon$
\Ensure $\bfx$
    \For{$k=0,1,\dots$ until convergence}
    \State $\bfu^{(k)} = \bfL\bfx^{(k)}$
    \State $\bfw_{reg}^{(k)} = \bfu^{(k)} \left(1-\left(\frac{\varepsilon^2}{\left(\bfu^{(k)}\right)^2+\varepsilon^2}\right)^{\frac{1}{2}}\right)$
    \State Estimate $\lambda^{(k)}$
    \State $\bfx^{(k+1)} = \argmin_{\bfx}\left\{ \frac{1}{2}\|\bfA\bfx-\bfb\|^2_2 + \frac{\left(\lambda^{(k)}\right)^2}{2}\|\bfL\bfx-\bfw_{reg}^{(k)}\|^2_2\right\}$
    \EndFor
\end{algorithmic}
\end{algorithm}

\subsection{The inner minimization problem}\label{subsec:inner}
At each iteration $k$ in both SB and MM we solve a problem of the form as given by \cref{eq:genTik_hk} with the regularization parameter $\lambda$ replaced by $\lambda^{(k)}$
(see \cref{eq:Breg3} for SB and \cref{eq:MMsolve} for MM). In SB, $\bfh^{(k)} = \bfd^{(k)}-\bfg^{(k)}$, while in MM, $\bfh^{(k)} = \bfw^{(k)}_{reg}$.
We focus on \cref{eq:genTik_hk} and regularization parameter estimation methods for finding $\lambda^{(k)}$ that can be used within each iteration of SB and MM.

\subsection{Matrix decompositions}\label{subsec:GSVD}
 To rewrite \cref{eq:genTik_hk}, we will use the generalized singular value decomposition (GSVD)~\cite{van1983matrix,hansen1998rank,howland2004generalizing}, which is a joint matrix decomposition of two matrices $\bfA$ and $\bfL$. 
Consider $\bfA \in \mathbb{R}^{m \times n}$ and $\bfL \in \mathbb{R}^{p \times n}$, and let $\tilde{n} = \text{rank}(\bfL)$.
When $m\geq n$ and the invertibility condition \cref{eq:invcond} is satisfied, there exist orthogonal matrices $\bfU \in \mathbb{R}^{m \times m}$ and $\bfV \in \mathbb{R}^{p \times p}$ and an invertible matrix $\bfX \in \mathbb{R}^{n \times n}$ such that 
\begin{equation}\label{gsvd}
    \bfA = \bfU \tbfUpsilon \bfX^{-1}, \quad \bfL = \bfV \tbfM \bfX^{-1}, 
\end{equation}
where 
\begin{align*}
    \tbfUpsilon &= \begin{bmatrix} \bfUpsilon & \bfzero_{\tilde{n}\times(n-\tilde{n})} \\ \bfzero_{(n-\tilde{n})\times\tilde{n}} &
        \bfI_{(n-\tilde{n})\times (n-\tilde{n})} \\
        {\bfzero}_{(m-n)\times\tilde{n}} & {\bfzero}_{(m-n) \times(n-\tilde{n})}
    \end{bmatrix},  \qquad \bfUpsilon = \text{diag}(\upsilon_1,\dots, \upsilon_{\tilde{n}}), \nonumber \\
    \tbfM &= \begin{bmatrix}
        \bfM & {\bfzero}_{\tilde{n} \times (n-\tilde{n})} \\ {\bfzero}_{(p-\tilde{n}) \times \tilde{n}} & {\bfzero}_{(p-\tilde{n}) \times (n-\tilde{n})}
    \end{bmatrix}, \qquad \bfM = \text{diag}(\mu_1,\dots,\mu_{\tilde{n}}), 
\end{align*}
with  $0 \leq \upsilon_1 \leq \cdots \leq \upsilon_{\tilde{n}} < 1$, $1 \geq \mu_1 \geq \cdots \geq \mu_{\tilde{n}} > 0$, and
$\upsilon^2_i + \mu^2_i = 1$   for  $i=1,\dots, \tilde{n}$.  For $i=1,\dots,\tilde{n}$, $\gamma_i = \upsilon_i/\mu_i$ are called the generalized singular values.
Here, and throughout, we use  $\bfI_{a \times b}$ and ${\bf {0}}_{a \times b}$ to denote the identity and zero matrices, respectively,  of dimension $a\times b$ and which may possibly have no rows or no columns. 

The GSVD is helpful for analyzing the properties of the solutions as a function of $\lambda$ and, consequently, for analyzing regularization methods. On the other hand, computationally, it is only helpful for small problems. Other joint decompositions, such as the discrete Fourier transform or the discrete cosine transform, can be practical for larger problems, as we will see in the numerical results section. 

\section{Estimation of the regularization parameter}\label{sec:reg}
In this section, we present three new methods for finding the regularization parameter in the iterative updates for the SB and MM methods. We consider in \cref{subsec:GCV} the GCV method and in \cref{subsec:chi2} the $\chi^2$ dof test. Note that while the $\chi^2$  dof test requires that $\Cb$ is known, there is no such requirement for the GCV method. We also consider in \cref{subsubsec:noncentralchi} a non-central $\chi^2$ formulation adapted to the iteration in the SB algorithm.  Subsection \ref{subsec:othermethods} includes two other parameter selection methods, DP and RWP, that will be used for comparison in the numerical results section.

\subsection{The method of generalized cross validation}\label{subsec:GCV}
GCV \cite{aster2018parameter, golub1979generalized} is a parameter selection method that selects $\lambda$ to minimize predictive risk.  GCV has been applied to the generalized Tikhonov problem \cref{eq:GenTik},
which has the solution   $\bfx_{\lambda} = \bfA_\lambda^\sharp \bfb$,
where $\bfA_\lambda^\sharp$ is the influence matrix defined by $\bfA_\lambda^\sharp= \bfA_\bfL^{-1}\bfA^\top$ with $\bfA_\bfL = \bfA^\top\bfA + \lambda^2\bfL^\top\bfL$.
In GCV, $\lambda$ is selected to minimize the GCV function
\begin{align}
    G(\lambda) = \frac{\norm{\bfA \bfx_\lambda - \bfb}_2^2}{\left[\text{Tr}\left(\bfI -  {\bfA}_\lambda\right)\right]^2} \label{GCVGen},
\end{align}
where ${\bfA}_\lambda = \bfA\bfA_\lambda^\sharp$ is the resolution matrix.
GCV has also been extended to other regularization problems \cite{ChEaOl:11}.
Formulae for the GCV to solve~\cref{eq:GenTik} in terms of the GSVD of $\{\bfA,\bfL\}$ are given in \cite{chung2014optimal, mead2008newton, vatankhah2014regularization} for different orderings of the sizes for $m,n,p$. 

In \cite{buccini2022comparison}, GCV is used to select the parameter at each iteration in MM when adaptive majorants are used. With adaptive majorants, the minimization problem has the same form as \cref{eq:GenTik}.  We will extend GCV to the case when the inner-minimization problem has the form \cref{eq:genTik_hk} for which the solution $\bfx_\lambda$ is different.  This also impacts the formula in terms of the GSVD of $\{\bfA,\bfL\}$, and impacts the form of the GCV function as summarized in the following Theorem.  The proof follows the steps given in \cite{ChEaOl:11}, but modified due to the vector $\bfh^{(k)}$.
\begin{theorem}\label{Thm:gcv}
     The GCV function for \cref{eq:genTik_hk}, where $\bfh^{(k)}$ is replaced by $\bfh$, has the same form as \cref{GCVGen}, but now with $\bfx_\lambda$ given by 
     \begin{align}
    \bfx_{\lambda} &= \bfA_\bfL^{-1}(\bfA^\top \bfb + \lambda^2 \bfL^\top \bfh) =
\bfA_\lambda^\sharp \bfb + \bfL^\sharp_\lambda \bfh,
    \label{xsolLh}
\end{align}
where 
$\bfL^\sharp_\lambda = \lambda^2\bfA_\bfL^{-1}\bfL^\top$.  When $m \geq n$ and the GSVD is defined as in \cref{gsvd}, $G(\lambda)$ for $n \geq p$ and $p>n$ is given by
 \begin{align}
G(\lambda) 
    =  \frac{\sum_{i=1}^{\tilde{n}} \left(\frac{\lambda^2\gamma_i(\bfv_i^\top \bfh)}{\gamma_i^2+\lambda^2}\right)^2 + \sum_{i=1}^{\tilde{n}} \left(\frac{\lambda^2(\bfu_i^\top\bfb)}{\gamma_i^2+\lambda^2}\right)^2 + \sum_{i=n+1}^m \left(\bfu_i^\top \bfb\right)^2- 2\sum_{i=1}^{\tilde{n}} \frac{\lambda^4 \gamma_i (\bfu_i^\top \bfb) (\bfv_i^\top \bfh)}{(\gamma_i^2+\lambda^2)^2}
}{\left[\max(m-n,0) + \sum_{i=1}^{\tilde{n}} \frac{\lambda^2}{\gamma_i^2+\lambda^2}\right]^2}  \label{eq:GCV_GSVDRR}, 
\end{align}
where we ignore any sum in which the lower limit is greater than the upper limit.
\end{theorem}

\begin{proof}
In GCV, $\lambda$ is selected to minimize the average predictive risk when we leave out an entry of $\bfb$.  Let $\bfx_\lambda^{[k]}$ be the solution when the $k^{\text{th}}$ entry of $\bfb$ is missing.
Then, in GCV we select $\lambda$ to minimize the predictive risk for all $k$:
\begin{align*}
    \min_{\lambda} G(\lambda) = \min_{\lambda}\left\{\frac{1}{m}\sum_{k=1}^m ((\bfA \bfx^{[k]}_\lambda)_k - b_k)^2\right\}. \label{GCVfunc}
\end{align*}
Defining the $m \times m$ matrix
    $\bfE_k = \text{diag}(1,1,\dots,1,0,1,\dots,1)$,
where the $k^{\text{th}}$ entry is $0$, the solution $\bfx_\lambda^{[k]}$ can be written as
\begin{align}\label{eq:xk}
    \bfx_\lambda^{[k]} = (\bfA^\top\bfE_k^\top \bfE_k\bfA + \lambda^2 \bfL^\top \bfL)^{-1}(\bfA^\top \bfE_k^\top \bfE_k \bfb + \lambda^2\bfL^\top \bfh).
\end{align}
From the definition of $\bfE_k$, we have the following properties \cite{ChEaOl:11}:
\begin{align}
    \bfE^\top_k \bfE_k = \bfE_k,
    \text{  and   } \bfE_k = \bfI-\bfe_k \bfe_k^\top, \label{eq:prop1}
\end{align}
where $\bfe_k$ is the $k^{\text{th}}$ unit column vector of length $m$.
From these properties, we obtain 
\begin{equation*}
    \bfA^\top\bfE_k^\top \bfE_k\bfA + \lambda^2 \bfL^\top \bfL = (\bfA^\top\bfA + \lambda^2 \bfL^\top \bfL) - \bfa_k \bfa_k^\top = \bfA_\bfL - \bfa_k \bfa_k^\top,
\end{equation*}
where $\bfa_k^\top = \bfe_k^\top \bfA$ is the $k^{\text{th}}$ row of $\bfA$.  
Then, by applying the Sherman-Morrison formula 
$$(\bfB +\bfu\bfv^\top)^{-1} = \bfB^{-1}-\frac{\bfB^{-1}\bfu\bfv^\top\bfB^{-1}}{1 + \bfv^\top \bfB^{-1} \bfu}$$
to $\bfB = \bfA_\bfL$, $\bfu=-\bfa_k$, and $\bfv=\bfa_k$, and using the first     property in \cref{eq:prop1}, we rewrite \cref{eq:xk} as
\begin{align}
    \bfx_\lambda^{[k]} &=\left(\bfA_\bfL^{-1} + \frac{\bfA_\bfL^{-1}\bfa_k\bfa^\top_k\bfA_\bfL^{-1}}{1-\bfa_k^\top\bfA_\bfL^{-1}\bfa_k} \right)(\bfA^\top \bfE_k \bfb + \lambda^2\bfL^\top \bfh) \nonumber \\
    &= \left( \bfI + \frac{\bfA_\bfL^{-1}\bfa_k\bfa^\top_k}{1-\bfa_k^\top\bfA_\bfL^{-1}\bfa_k} \right)\bfA_\bfL^{-1}(\bfA^\top \bfE_k \bfb + \lambda^2\bfL^\top \bfh). \label{eq:xlambdaGCV}
\end{align}
Notice that $\bfa_k^\top\bfA_\bfL^{-1}\bfa_k$ is the $k^{\text{th}}$ diagonal entry of $\tilde{\bfA}_\lambda$, which we will denote by $\tilde{a}_{kk}$. Since $(\bfA \bfx_\lambda^{[k]})_k = \bfe^\top_k \bfA \bfx_\lambda^{[k]} = \bfa_{k}^\top \bfx_\lambda^{[k]}$, we use \cref{eq:xlambdaGCV} to obtain 
\begin{align}
    (\bfA \bfx_\lambda^{[k]})_k &=
    \left( \bfa_k^\top + \frac{\bfa_k^\top\bfA_\bfL^{-1}\bfa_k\bfa^\top_k}{1-\bfa_k^\top\bfA_\bfL^{-1}\bfa_k} \right)\bfA_\bfL^{-1}(\bfA^\top \bfE_k \bfb + \lambda^2\bfL^\top \bfh) \nonumber \\
    &=\left(1+\frac{\tilde{a}_{kk}}{1-\tilde{a}_{kk}}\right)\bfa_k^\top\bfA_\bfL^{-1}(\bfA^\top \bfE_k \bfb + \lambda^2\bfL^\top \bfh) \nonumber \\
    &=\left(\frac{1}{1-\tilde{a}_{kk}}\right)\bfe_k^\top\bfA\bfA_\bfL^{-1}(\bfA^\top \bfE_k \bfb + \lambda^2\bfL^\top \bfh) \nonumber \\
    &=\left(\frac{1}{1-\tilde{a}_{kk}}\right)\bfe_k^\top\bfA\bfA_\lambda^\sharp \bfE_k \bfb + \left(\frac{1}{1-\tilde{a}_{kk}}\right)\bfe^\top_k\bfA\bfL^\sharp_\lambda \bfh. \nonumber 
\end{align}
Then, by the second property in \cref{eq:prop1}  
\begin{align}
    (\bfA \bfx_\lambda^{[k]})_k - b_k &= \left(\frac{1}{1-\tilde{a}_{kk}}\right)\bfe_k^\top\bfA\bfA_\lambda^\sharp \bfE_k \bfb + \left(\frac{1}{1-\tilde{a}_{kk}}\right)\bfe^\top_k\bfA\bfL^\sharp_\lambda \bfh -\bfe^\top_k \bfb \nonumber \\
    &= \left(\frac{1}{1-\tilde{a}_{kk}}\right)\bfe_k^\top\bfA\bfA_\lambda^\sharp \bfb  \nonumber \\
     & \qquad \qquad \qquad - \left(\frac{\tilde{a}_{kk}}{1-\tilde{a}_{kk}}\right)\bfe_k^\top \bfb + \left(\frac{1}{1-\tilde{a}_{kk}}\right)\bfe^\top_k\bfA\bfL^\sharp_\lambda \bfh -\bfe^\top_k \bfb \nonumber  \\
    &= \left(\frac{1}{1-\tilde{a}_{kk}}\right)\bfe_k^\top\bfA\bfA_\lambda^\sharp \bfb + \left(\frac{1}{1-\tilde{a}_{kk}}\right)\bfe^\top_k\bfA\bfL^\sharp_\lambda \bfh -\left(1+\frac{\tilde{a}_{kk}}{1-\tilde{a}_{kk}}\right)\bfe^\top_k \bfb \nonumber \\
    &= \left(\frac{1}{1-\tilde{a}_{kk}}\right)\bfe_k^\top\left(\bfA\bfx_\lambda - \bfb\right)  
    = \frac{(\bfA\bfx_\lambda)_k - b_k}{1-\tilde{a}_{kk}}. \nonumber
\end{align}
Therefore, the GCV function is given by
\begin{align}
    G(\lambda) &= \frac{1}{m} \sum_{k=1}^m \left[ \frac{(\bfA\bfx_\lambda)_k - b_k}{1-\tilde{a}_{kk}}\right]^2. \nonumber 
    \end{align}
We can approximate the diagonal values of $\tilde{\bfA}_\lambda$ by the average diagonal value $\text{Tr}(\tilde{\bfA}_\lambda)/m$, which produces a weighted version of the function \cite{golub1979generalized}.  The resulting function is then  
    \begin{align*}
   G(\lambda) &= \frac{\norm{\bfA \bfx_\lambda - \bfb}_2^2}{\left[\text{Tr}\left(\bfI - \tilde{\bfA}_\lambda\right)\right]^2},
\end{align*}
which is the desired GCV function.

Next, we derive the formula of the GCV function \cref{eq:GCV_GSVDRR} in terms of the GSVD of $\{\bfA,\bfL\}$ when $m \geq n$, given in \cref{gsvd}. To do so, we first write       
    \begin{align}
        \tilde{\bfA}_\lambda = \bfA \bfA_\bfL^{-1}\bfA^\top = \bfU\tbfUpsilon \boldsymbol{\Phi}^{-1}\tbfUpsilon^\top\bfU^\top, \label{eq:AAt}
    \end{align} where $\boldsymbol{\Phi} = (\tbfUpsilon^\top\tbfUpsilon + \lambda^2\tbfM^\top \tbfM)$.
    The numerator then becomes
    \begin{align}\label{eq:GCVnum1}
        \|\bfA \bfx_\lambda - \bfb\|_2^2 &= \|\bfA \bfA_\lambda^\sharp \bfb + \bfA\bfL_\lambda^\sharp \bfh - \bfb\|_2^2 \nonumber \\
        &=\|\bfU \tbfUpsilon \boldsymbol{\Phi}^{-1}\tbfUpsilon^\top \bfU^\top \bfb + \lambda^2\bfU \tbfUpsilon \boldsymbol{\Phi}^{-1}\tbfM^\top \bfV^\top \bfh-\bfb\|_2^2. 
    \end{align}
        Factoring out the orthogonal matrix $\bfU$ from \cref{eq:GCVnum1}, we obtain
    \begin{align}
        &\|\bfU \tbfUpsilon \boldsymbol{\Phi}^{-1}\tbfUpsilon^\top \bfU^\top \bfb + \lambda^2\bfU \tbfUpsilon \boldsymbol{\Phi}^{-1}\tbfM^\top \bfV^\top \bfh-\bfb\|_2^2 \nonumber \\
        & \qquad \qquad \qquad \qquad  = 
        \| (\tbfUpsilon \boldsymbol{\Phi}^{-1}\tbfUpsilon^\top - \bfI_m) \bfU^\top \bfb + \lambda^2 \tbfUpsilon \boldsymbol{\Phi}^{-1}\tbfM^\top \bfV^\top \bfh\|_2^2 \nonumber \\
        & \qquad \qquad \qquad \qquad  = 
        \| \bfK \bfU^\top \bfb + \lambda^2 \bfZ \bfV^\top \bfh\|_2^2 \label{eq:gcvnumnorm},
    \end{align}
where $\bfK \in \mathbb{R}^{m \times m}$ and $\bfZ\in \mathbb{R}^{m \times p}$ are diagonal matrices, defined as
    \begin{align}
        \bfK = \begin{bmatrix}
            -\bfPsi & {\bfzero} & {\bfzero} \\
            {\bfzero} & {\bfzero}_{n-\tilde{n}} & {\bfzero} \\
            {\bfzero}&{\bfzero}&-{\bf{I}}_{m-n}\end{bmatrix} \text{ and } \bfZ = \begin{bmatrix}
            \bfzeta & {\bfzero} \\{\bfzero}&{\bfzero}_{(m-\tilde{n}) \times (p-\tilde{n})} \label{eq:diagups}
        \end{bmatrix}
    \end{align}
    with $\bfPsi,\bfzeta \in \mathbb{R}^{\tilde{n} \times \tilde{n}}$ given by 
    \begin{align}
        \bfPsi  
    = \text{diag}\left(\frac{\lambda^2}{\gamma_1^2+\lambda^2}, \dots, \frac{\lambda^2}{\gamma_{\tilde{n}}^2+\lambda^2}\right) \mbox{ and } \bfzeta  
    = \text{diag}\left(\frac{\gamma_1}{\gamma_1^2+\lambda^2},\dots,\frac{\gamma_{\tilde{n}}}{\gamma_{\tilde{n}}^2+\lambda^2}\right). \label{psi}
    \end{align} 
This can be written out as
    \begin{align}
           \| \bfK \bfU^\top \bfb + \lambda^2 \bfZ \bfV^\top \bfh\|_2^2
           &= \sum_{i=1}^{\tilde{n}} \left(\frac{\lambda^2\gamma_i(\bfv_i^\top \bfh)}{\gamma_i^2+\lambda^2}\right)^2 + \sum_{i=1}^{\tilde{n}} \left(\frac{\lambda^2(\bfu_i^\top\bfb)}{\gamma_i^2+\lambda^2}\right)^2 + \sum_{i=n+1}^m \left(\bfu_i^\top \bfb\right)^2 \label{eq:GCVnumer} \\
           & \qquad \qquad- 2\sum_{i=1}^{\tilde{n}} \frac{\lambda^4 \gamma_i (\bfu_i^\top \bfb) (\bfv_i^\top \bfh)}{(\gamma_i^2+\lambda^2)^2},  \, \nonumber
    \end{align}
    where we ignore any sum where the lower limit is greater than the upper limit. In \cref{eq:GCVnumer} we deliberately show the cross terms involving the diagonal matrix-vector products in the numerator.  Practically, we use the form given in \cref{eq:gcvnumnorm} that can be computed efficiently.

    In the denominator, we can use \cref{eq:AAt}  
    to obtain the desired result,  
    \begin{align}
        \left[\text{Tr}\left(\bfI - \tilde{\bfA}_\lambda\right)\right]^2 &=\left[\text{Tr}\left(\bfI - \bfU\tbfUpsilon \boldsymbol{\Phi}^{-1}\tbfUpsilon^\top\bfU^\top\right)\right]^2 \nonumber \\
        &= \left[\text{Tr}\left(\bfI - \tbfUpsilon \boldsymbol{\Phi}^{-1}\tbfUpsilon^\top\right)\right]^2 \nonumber 
        =\left[\sum_{i=1}^{\tilde{n}} \frac{\lambda^2}{\gamma_i^2+\lambda^2} + \max(m-n,0) \right]^2. \label{eq:GCVdenom}
    \end{align}
\end{proof}

To find the minimizer of $G(\lambda)$, following the approach in \cite{Regtools}, we first evaluate $G$ at $200$ logarithmically-spaced values of $\lambda$. We then take the $\lambda$ that gives the smallest value of $G$ and search for the minimizer in an interval around this value.

\subsection{The \texorpdfstring{$\chi^2$} {chi-squared} degrees of freedom test}\label{subsec:chi2}
 At each iteration in SB and MM, we minimize the functional $J(\bfx)$ in \cref{eq:genTik_hk}. The $\chi^2$ dof test is a parameter selection method that treats $\bfx$ as a random variable and utilizes the distribution of the functional evaluated at its minimizer to select $\lambda$ \cite{mead2008parameter}.  In \cite{mead2008newton}, the $\chi^2$ dof test is applied to the functional $J_L(\bfx)$
 which is defined by
\begin{align}\label{eq:chirls1}
   J_L(\bfx) = \|\bfA \bfx - \bfb\|_2^2 + \norm{\bfL(\bfx - \bfx_0)}_{\bfW_\bfh}^2.
\end{align}
Here, $\bfx_0$ is a reference vector of prior information on $\bfx$.  We assume that $\bfx_0 = \overline{\bfx}$, where $\overline{\bfx}$ is the expected value of $\bfx$. 
It is also assumed in \cite{mead2008newton} that  $\bfL$ has fewer rows than columns, $p\le n$, and 
$\bfL(\bfx-\bfx_0) = \boldsymbol{\epsilon}_\bfh$, where 
$ \boldsymbol{\epsilon}_\bfh\sim  \mathcal{N}(\bfzero,\bfC_\bfh)$ and $\bfC_\bfh$ is a SPD covariance matrix.
The matrix $\bfW_\bfh = \bfC_\bfh^{-1}$ is then the corresponding precision matrix. We also suppose that the model errors $\bfx-\bfx_0 = \boldsymbol{\epsilon}_\bfL$
are normally distributed with covariance $\bfC_\bfL$.
From this, $\bfb \sim \mathcal{N}(\bfA\overline{\bfx}, \bfI_m + \bfA\bfC_\bfL\bfA^\top)$. 
Then, in \cite[Theorem 3.1]{mead2008newton}, it is shown that for large $m$ and for given $\bfW_\bfh$, $J_L$ 
is a random variable that follows a $\chi^2$ distribution with $m+p-n$ degrees of freedom.

The goal of the $\chi^2$ method is then to find $\bfW_{\bfh}$ so that $J_L(\bfx)$ 
most resembles a $\chi^2$ distribution with the appropriate degrees of freedom.   
In \cite{renaut2010regularization}, the $\chi^2$ dof test for \cref{eq:chirls1} is extended to the case when $\bfx_0 \neq \overline{\bfx} $, in which case $J_L(\bfx)$ follows a non-central $\chi^2$ distribution.

Under the assumption that $\bfC_\bfh = \lambda^{-2}\bfI_p$, we define
\begin{align}\label{eq:chirls2}
    \bfx_\lambda = \argmin_{\bfx} J_L(\bfx) = \argmin_{\bfx} \left\{\norm{\bfA \bfx - \bfb}_2^2 + \lambda^2\norm{\bfL(\bfx - \bfx_0)}_2^2\right\}.
\end{align}
$J_L(\bfx)$ then follows a $\chi^2$ distribution. In particular, the following theorem gives its degrees of freedom, that holds for $p>n$ and $p\le n$. 
\begin{theorem}\label{Thm:chi2}
     Suppose that $\bfA\bfx - \bfb \sim \mathcal{N}({\bfzero},\bfI_m)$, $\bfL(\bfx-\bfx_0) \sim \mathcal{N}({\bfzero},\lambda^{-2}\bfI_p)$,  $m \geq n$, and $\tilde{n}=\text{rank}(\bfL)$. For a given $\lambda$ and the corresponding solution $\bfx_\lambda$, $J_L(\bfx_\lambda)$ is a random variable that follows a $\chi^2$ distribution with  $\tilde{m} = \tilde{n} + \max{(m-n,0)}$ degrees of freedom.
\end{theorem}
\begin{proof}
The case when $m \geq n \geq p$ is given in \cite[Theorem 3.1]{mead2008newton}.
For the general case when $m \geq n$ without a restriction on $p$, the approach follows that given in the proof of \cite[Theorem 3.1]{mead2008newton}.
Noting that the solution of \cref{eq:chirls2} is given by
$\bfx_\lambda = \bfx_0 + \bfA_\bfL^{-1}\bfA^\top\bfr$, where $\bfr = \bfb-\bfA\bfx_0$, we can use the GSVD of $\{\bfA,\bfL\}$ in \cref{gsvd}, to write $J_L(\bfx_\lambda)$ as
\begin{align}
    J_L(\bfx_\lambda) &=
    \|\bfA\bfx_0 + \bfA \bfA_\bfL^{-1} \bfA^\top \bfr - \bfb\|_2^2 + \lambda^2\|\bfL \bfA_\bfL^{-1}\bfA^\top\bfr\|^2_2 \nonumber \\ 
    &=\norm{\bfA \bfA_\bfL^{-1} \bfA^\top \bfr - \bfr}_2^2 + \lambda^2\|\bfL \bfA_\bfL^{-1}\bfA^\top\bfr\|^2_2 \nonumber \\
    &= \|\bfU\tbfUpsilon\boldsymbol{\Phi}^{-1}\tbfUpsilon^\top \bfU^\top \bfr - \bfr\|_2^2 + \lambda^2\|\bfV\tbfM \boldsymbol{\Phi}^{-1} \tbfUpsilon^\top \bfU^\top\bfr\|^2_2 \nonumber \\
    \label{eq:chiP} &= \|(\tbfUpsilon\boldsymbol{\Phi}^{-1}\tbfUpsilon^\top-\bfI_m) \bfs\|_2^2 + \lambda^2\|\tbfM \boldsymbol{\Phi}^{-1} \tbfUpsilon^\top \bfs\|^2_2,
\end{align}
where $\bfs=[s_1,\dots,s_m]^\top=\bfU^\top\bfr$.
We note that $\tbfUpsilon \boldsymbol{\Phi}^{-1}\tbfUpsilon^\top - \bfI_m = \bfK$ and
    \begin{align}
    \tbfM \boldsymbol{\Phi}^{-1} \tbfUpsilon^\top = \begin{bmatrix}
            \bfzeta & {\bfzero} \\{\bfzero}&{\bfzero}_{(p-\tilde{n}) \times (m-\tilde{n})} \nonumber
        \end{bmatrix},
    \end{align}
where $\bfK$ is defined in \cref{eq:diagups} and $\bfzeta \in \mathbb{R}^{\tilde{n} \times \tilde{n}}$ is defined in \cref{psi}.

Thus, \cref{eq:chiP} becomes
\begin{align}
    J_L(\bfx_\lambda)&= \sum_{i=1}^{\tilde{n}} \frac{\lambda^4}{(\gamma_i^2 + \lambda^2)^2}s_i^2 + \sum_{i=n+1}^{m} s_i^2 + \sum_{i=1}^{\tilde{n}} \frac{\lambda^2\gamma^2_i}{(\gamma_i^2 + \lambda^2)^2}s_i^2 \nonumber \\
    &= \sum_{i=1}^{\tilde{n}} \frac{\lambda^2}{\gamma_i^2 + \lambda^2}s_i^2 + \sum_{i=n+1}^{m} s_i^2
    = \sum_{i=1}^{\tilde{n}} k_i^2 + \sum_{i=n+1}^{m} k_i^2,  \label{eq:chi2mp}
\end{align}
with $\bfk = [k_1,\dots,k_m]^\top = \bfT\bfs$, where
\begin{align*}
    \bfT = \begin{bmatrix}
        \bfPsi^{1/2} & {\bfzero} & {\bfzero} \\
        {\bfzero} & \bfI_{n-\tilde{n}} & {\bfzero} \\
        {\bfzero} & {\bfzero} & \bfI_{m-n}
    \end{bmatrix}.
\end{align*}
Next, we need to show that the $k_i$ are distributed standard normal.  Using the GSVD, we have that $\bfV\tilde{\bfM}\bfX^{-1}(\bfx-\bfx_0) \sim \mathcal{N}({\bfzero}, \lambda^{-2}\bfI_p)$.  Using the properties for normally distributed vectors, we obtain $\bfx-\bfx_0 \sim \mathcal{N}({\bfzero}, \lambda^{-2}\bfX\diag{(\bfM^{-2},{\bfzero}_{n-\tilde{n}})}\bfX^{-1})$.  The residual $\bfr$ then has mean zero and covariance matrix 
\begin{align}
    \bfI_m + \lambda^{-2}\bfU\bfUpsilon\begin{bmatrix}
        \bfM^{-2} & {\bfzero} \\ {\bfzero} & {\bfzero}_{n-\tilde{n}}
    \end{bmatrix}\bfUpsilon^\top\bfU^\top= \bfU\begin{bmatrix}
        \bfPsi^{-1} & {\bfzero} & {\bfzero} \\
        {\bfzero} & \bfI_{n-\tilde{n}} & {\bfzero} \\
        {\bfzero} & {\bfzero} & \bfI_{m-n}
    \end{bmatrix}\bfU^{\top}=\bfU\bfT^{-2}\bfU^\top.\nonumber
\end{align}
Thus, $\bfk$ has covariance $\bfT\bfU^\top(\bfU\bfT^{-2}\bfU^{\top})\bfU \bfT^\top = \bfI_m$ and the $k_i$ follow a standard normal distribution. Hence, $J_L(\bfx_\lambda)$ is the sum of $\tilde{m} = \tilde{n} + \max{(m-n,0)}$ squared standard normal random variables.
\end{proof}

When $\tilde{m}$ is large, the $\chi^2_{\tilde{m}}$ distribution can be approximated by a normal distribution with mean $\tilde{m}$ and variance $2\tilde{m}$, denoted by $\mathcal{N}(\tilde{m},2\tilde{m})$.  We can use this approximation to find $\lambda$ such that $J_L(\bfx_\lambda)$ is then approximately distributed as $\mathcal{N}(\tilde{m},2\tilde{m})$.
As in \cite{mead2008newton}, we form a $(1-\alpha)$ confidence interval to find $\lambda$:
\begin{align*}
    \tilde{m} - z_{\alpha/2}\sqrt{2\tilde{m}} \leq J_L(\bfx_\lambda) \leq \tilde{m}+z_{\alpha/2}\sqrt{2\tilde{m}}.
\end{align*}
Here, $z_{\alpha/2}$ is the relevant $z$-value for the standard normal distribution, which defines the bounds for the $(1-\alpha)$ confidence interval.
Defining $F(\lambda) = J_L(\bfx_\lambda) - \tilde{m}$, we   apply Newton's method to $\lambda$ such that 
\begin{align} \label{eq:NMint}
    - z_{\alpha/2}\sqrt{2\tilde{m}} \leq F(\lambda) \leq z_{\alpha/2}\sqrt{2\tilde{m}}.
\end{align}
The derivative of $F(\lambda)$ is $F^\prime(\lambda) = 2\lambda\norm{\tilde{\bfs}}_2^2$.
Here, $\tilde\bfs = [\tilde s_1,\dots,\tilde s_{\tilde{n}}]^\top$, where $\tilde{s}_i = {\gamma_i s_i}/(\gamma_i^2 + \lambda^2)$
for $i=1,\dots, \tilde{n}$. 
We notice immediately that for $\lambda>0$, as required, $F^\prime>0$, and hence it is reasonable to use a root-finding method for $\lambda$, provided that we can find an interval in which $F(\lambda)$ changes sign. This is discussed in more detail in \cite{mead2008newton}.  An example of $F(\lambda)$ is given in \cref{fig:FvFC:a}.

The tolerance for Newton's method depends on the selection of the significance level $\alpha$.  A smaller value of $\alpha$ makes the $(1-\alpha)$ confidence interval wider, increasing the probability that a random value of $\lambda$ satisfies \cref{eq:NMint}.  On the other hand, a larger $\alpha$ narrows this interval, increasing the confidence that $\lambda$ satisfies the $\chi^2$ distribution.

For SB and MM, the 
inner minimization problem \cref{eq:genTik_hk} is not of the desired form \cref{eq:chirls1} for the central $\chi^2$ dof test. In particular, we need $\bfL\bfx_0$ in place of $\bfh^{(k)}$ in \cref{eq:genTik_hk}. To find $\bfx_0$ such that $\bfh^{(k)} \approx \bfL\bfx_0$, we will use the the $\bfA$-weighted generalized inverse of $\bfL$, denoted $\bfL^\dagger_{\bfA}$ \cite{elden1982weighted}.  In terms of the GSVD of $\{\bfA,\bfL\}$, $\bfL^\dagger_{\bfA} = \bfX\tbfM^\dagger \bfV^\top$,
 where $\tbfM^\dagger \in \mathbb{R}^{n \times p}$ has diagonal entries ${1}/{\mu_i}$.  If $p \geq n$, then $\bfL_{\bfA}^\dagger=\bfL^\dagger$, where $\bfL^\dagger$ is the pseudoinverse of $\bfL$ \cite{hansen1998rank}. 
 For $p<n$, $\bfL_{\bfA}^\dagger \neq \bfL^\dagger$ in general.  

In terms of the GSVD of $\{\bfA, \bfL\}$, $\bfL\bfL^\dagger_{\bfA} = \bfV\tbfM \tbfM^\dagger \bfV^\top$. When $\bfL$ is invertible, and when 
$p \leq n$, $\tbfM \tbfM^\dagger = \bfI_p$ so $\bfL\bfL^\dagger_{\bfA}=\bfI_p$. In the $\chi^2$ dof test, we will estimate $\bfx_0$ as $\bfL^\dagger_{\bfA} \bfh^{(k)}$ and use $\bfL\bfL_{\bfA}^\dagger \bfh^{(k)}$ in place of $\bfh^{(k)}$.  With this, \cref{eq:genTik_hk} can be rewritten in the form of \cref{eq:chirls1}. Provided that $\bfx_0$ is approximately the expected value of $\bfx$, we may use the central $\chi^2$ dof test at each iteration of SB and MM.  In the case where $\bfx_0 \not \approx \overline{\bfx}$, the central $\chi^2$ dof test no longer applies.

\subsubsection{The non-central \texorpdfstring{$\chi^2$}{chi-squared} dof test}\label{subsubsec:noncentralchi}
When $\bfx_0 \not \approx \overline{\bfx}$,   $J_L(\bfx)$
follows a non-central $\chi^2$ distribution with $\tilde{m}$ degrees of freedom and non-centrality parameter $c$, defined by 
\begin{align*}
    c(\lambda) = \sum_{i=1}^{\tilde{n}} \frac{\lambda^2q_i^2}{\gamma_i^2 + \lambda^2} + \sum_{i=n+1}^m q_i^2, 
\end{align*}
where $\bfq=[q_1,\dots,q_m]^\top=\bfU^\top \bfA(\overline{\bfx} - \bfx_0)$ \cite{renaut2010regularization}. 
As with the central $\chi^2$ dof test, the goal of the non-central test is to find $\lambda$ so that $J_L \sim \chi^2(\tilde{m},c(\lambda))$. When $\tilde{m}$ is sufficiently large, this non-central $\chi^2$ distribution can be approximated by a normal distribution with mean $\tilde{m}+c(\lambda)$ and variance $2\tilde{m} + 4c(\lambda)$.  Using this approximation, we can form a $(1-\alpha)$ confidence interval to find $\lambda$:
\begin{align*}
    \tilde{m} + c(\lambda) - z_{\alpha/2}\sqrt{2\tilde{m}+4c(\lambda)} \leq J_L(\bfx_\lambda) \leq \tilde{m}+c(\lambda)+z_{\alpha/2}\sqrt{2\tilde{m}+4c(\lambda)}.
\end{align*}
Defining $F_C(\lambda) = J_L(\bfx_\lambda)- (\tilde{m} + c(\lambda))$, this is equivalent to solving the problem
\begin{align}\label{NM2}
    -z_{\alpha/2}\sqrt{2\tilde{m}+4c(\lambda)} \leq F_C(\lambda) \leq z_{\alpha/2}\sqrt{2\tilde{m}+4c(\lambda)}.
\end{align}
In this case, the  derivative of $F_C(\lambda)$   is
\begin{align}\label{eq:derivqtilde}
      F_C^\prime(\lambda)=  
      2\lambda\left(\norm{\tilde{\bfs}}_2^2 - \norm{\tilde{\bfq}}_2^2\right),
\end{align}
where $\tilde\bfq=[\tilde q_1,\dots,\tilde q_{\tilde{n}}]^\top$ with $\tilde{q}_i = {\gamma_i q_i}/(\gamma_i^2 + \lambda^2)$
for $i=1,\dots, \tilde{n}$. 
Notice that when $\bfx_0 = \overline{\bfx}$, $\bfq=\boldsymbol{0}$ and the non-central test reduces to the central test.

\begin{figure}
    \centering
\subfigure[$F(\lambda)$ \label{fig:FvFC:a}]{\includegraphics[width=3.65cm]{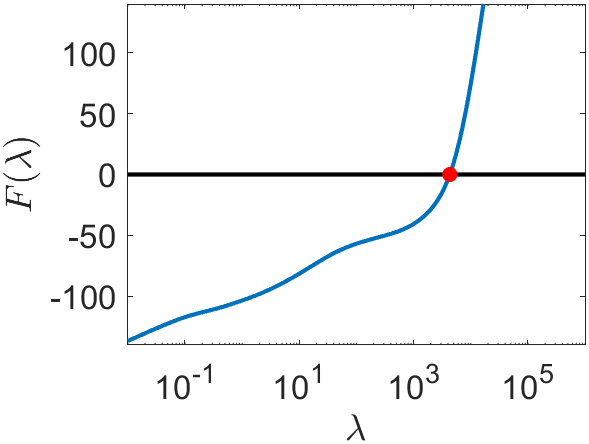}}
\subfigure[$F_C(\lambda)$ \label{fig:FvFC:b}]{\includegraphics[width=3.65cm]{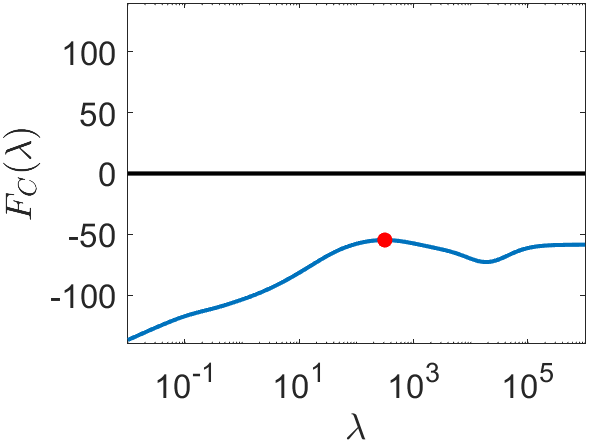}}
\caption{ Plots of $F(\lambda)$ and $F_C(\lambda)$ for MM, where the selected $\lambda$ is marked.  In \cref{fig:FvFC:b}, $F_C(\lambda)$ is not monotonic and does not have a root.}
    \label{fig:FvFC}
\end{figure}

From \cref{eq:derivqtilde}, it is immediate that  $F_C^\prime(\lambda)$ is not of constant sign, so $F_C(\lambda)$ may potentially have no root or multiple roots.  \Cref{fig:FvFC:b} shows an example where $F_C(\lambda)$ is not monotone and has no root.  
If $F_C(\lambda)$ has a single root, we can apply Newton's method to find $\lambda$ such that \cref{NM2} holds.
Otherwise, we follow the method presented in 
\cite{renaut2010regularization}
to find $\lambda$ such that $F_C(\lambda)$ is close to zero. 

To use the non-central $\chi^2$ dof test in SB and MM, we still rewrite \cref{eq:genTik_hk}
using $\bfL \bfL^\dagger_{\bfA} \bfh^{(k)}$ in place of $\bfh^{(k)}$.  We also need to have an estimate of the expected value of $\bfx$ to use as $\overline{\bfx}$ in the test.  Within both SB and MM, one option is to use the current solution $\bfx^{(k)}$ as an approximation of $\overline{\bfx}$.  Another option in SB is to estimate $\overline{\bfx}$ using $\bfL^\dagger_{\bfA} \bfd^{(k)}$ since $\bfd$ in SB approximates $\bfL\bfx$ quite well.  We will use $\overline{\bfx}=\bfx^{(k)}$, assuming that at a given step of the iteration we are looking to reduce the noise around the current estimate, and that for increasing $k$ with convergence we have $\overline{\bfx}=\lim_{k}\bfx^{(k)}$.

\subsection{Other Methods} \label{subsec:othermethods}
DP and RWP are other established methods that have been applied to \cref{eq:genTik_hk} at each iteration \cite{buccini2022comparison,etna_vol53_pp329-351}.  
Both methods utilize a closed form of the residual at each iteration to select $\lambda$.  

DP is a method that uses an estimate of the norm of the noise to bound the norm of the residual. Assuming that $\delta>0$ is an estimate of the norm of the noise and $\nu>1$ is a parameter, then $\lambda$ is selected so that
\begin{align}
    \norm{\bfA\bfx_\lambda - \bfb}_2 \leq \nu\delta. \label{eq:DP}
\end{align}
DP is applied at each iteration in MM in \cite{buccini2022comparison} and can be applied in its iterative form within other iterative methods.  Using the expression for the residual  $\bfr_\lambda = \bfA\bfx_\lambda-\bfb$, \cref{eq:DP} can be solved using a root-finding method. Since we normalize the noise, we will use $\delta = \norm{\bfe_\bfb}_2 = \sqrt{n}$ and set $\nu=1.01$.

RWP is a parameter selection method \cite{etna_vol53_pp329-351} for 2D problems that selects $\lambda$ so that the 2D residual most resembles white noise.  We take the residual $\bfr_\lambda$, reshape it into its 2D form $\bfR_\lambda \in \mathbb{R}^{N \times N}$, and use the normalized auto-correlation of $\bfR_\lambda$ to measure whiteness.  A non-negative scalar measure of the whiteness, $\mathcal{W}: \mathbb{R}^{N \times N} \to \mathbb{R}^+$, is then used to compare the residuals for different $\lambda$.  This measure is given by  
\begin{align*}
    \mathcal{W}(\lambda) = \norm{\rho(\boldsymbol{\mathcal{F}}\bfR_\lambda)} =  \frac{\norm{\boldsymbol{\mathcal{F}}\bfR_\lambda \star \boldsymbol{\mathcal{F}}\bfR_\lambda}_2^2}{\norm{\boldsymbol{\mathcal{F}}\bfR_\lambda}^4_2}=\frac{\sum_{l,m} |\left(\boldsymbol{\mathcal{F}}\bfR_\lambda\right)_{l,m}|^4}{\left(\sum_{l,m} |\boldsymbol{(\mathcal{F}}\bfR_\lambda)_{l,m}|^2\right)^2},
\end{align*}
where $\boldsymbol{\mathcal{F}} \in \mathbb{C}^{N \times N}$ is the 2D discrete Fourier transform matrix and $\star$ denotes the 2D correlation operator. The optimal $\lambda$ for the RWP is then $\lambda$ that minimizes $\mathcal{W}$. To minimize $\mathcal{W}$, we will use the same method that was used to minimize $G$ in GCV.  RWP is incorporated within ADMM in \cite{etna_vol53_pp329-351} while in \cite{buccini2022comparison}, it is applied to the residual at each iteration. As with DP, the expression for $\bfr_\lambda$ can be used to select $\lambda$ at each iteration in SB and MM.

\section{Numerical examples} \label{sec:Examples}
In this section, we apply SB and MM to two deblurring problems to test the parameter selection methods presented in \cref{sec:reg}.
To compare methods, we will use the relative error ($\text{RE}$) which is defined by
\begin{align*}\label{eq:RE}
    \text{RE} = \frac{\norm{\bfx-\bfx_{true}}_2}{\norm{\bfx_{true}}_2}
\end{align*}
and the improved signal to noise ratio (ISNR), which is defined by
\begin{align*}\label{eq:ISNR}
    \text{ISNR} = 20\log_{10}\left(\frac{\norm{\bfb-\bfx_{true}}_2}{\norm{\bfx-\bfx_{true}}_2}\right).
\end{align*}
Both $\text{RE}$ and $\text{ISNR}$ measure how close a solution is to $\bfx_{true}$, where a closer solution has a smaller $\text{RE}$ and a larger $\text{ISNR}$.

In practice, $\bfx_{true}$ is unknown, so to understand when these methods have converged, we will consider the relative change in $\bfx$ and $\lambda$.  The relative change in $\bfx$, defined as
\begin{align}\label{eq:RCx}
    \text{RC}(\bfx^{(k)}) =\frac{\norm{\bfx^{(k)}-\bfx^{(k-1)}}_2}{\norm{\bfx^{(k-1)}}_2},
\end{align}
measures how much the solution is changing at a given iteration.  If it is below a specified tolerance, $\text{TOL}_\bfx$, then the solution is not changing much and we can consider the solution to be converged and stop iterating.  This tolerance should depend on the noise level. Because we use the weighted data fidelity term with $\bfA$ and $\bfb$ as defined in \cref{eq:Reweighting}, it is sufficient to use a fixed tolerance, here $\text{TOL}_\bfx = 0.001$,  for all noise levels.

The relative change in $\lambda^2$,
\begin{align}\label{eq:RClambda}
    \text{RC}((\lambda^{(k)})^2) =\frac{|(\lambda^{(k)})^2 - (\lambda^{(k-1)})^2|}{|(\lambda^{(k-1)})^2|},
\end{align}
measures how much $\lambda^2$ changes from one iteration to the next.  If $\text{RC}((\lambda^{(k)})^2)<\text{TOL}_\lambda$,
then, we can stop selecting $\lambda$ and fix it at the current value until convergence. This reduces the computational cost from applying the parameter selection method at every iteration. In both \cref{eq:RCx,eq:RClambda}, we assume, without loss of generality, that the denominator is non-zero. 

For each example, we select $\lambda$ at each iteration with the different selection methods and compare the results to the \textit{optimal} fixed $\lambda$. 
To find this \textit{optimal} $\lambda$, we run each example to completion for $121$ values of $\lambda$ that are logarithmically spaced from $10^{-1}$ to $10^{3}$ and select as \textit{optimal} the $\lambda$ that has the smallest $\text{RE}$ at convergence. This means that the \textit{optimal} solution is not feasible in general as it requires knowledge of the true solution. We report the cost for the \textit{optimal} as running the method once with a fixed $\lambda$, but the actual cost is more expensive as it requires knowing the true solution and running the methods multiple times to find the \textit{optimal} $\lambda$. Moreover, $\tau$ and $\varepsilon$ are tuned to minimize the $\text{RE}$, requiring the knowledge of the true solution,  when the \textit{optimal} $\lambda$ is used.  When the parameter selection methods perform as well as using the fixed \textit{optimal} parameter, we should not expect there to be significant differences between solutions. In these cases, we instead focus on the cost of each of the methods to achieve these solutions, using the \textit{optimal} solution as a point of comparison.

For the $\chi^2$ dof tests,  we use $\bfx_0^{SB} = \bfL_{\bfA}^\dagger(\bfd-\bfg)$ and $\bfx_0^{MM} = \bfL_{\bfA}^\dagger(\bfw_{reg})$. For the non-central $\chi^2$ dof test, we use $\overline{\bfx} = \bfx^{(k)}$. 
Furthermore,  in all the numerical experiments we use a significance level of $\alpha = 0.999$ for the $\chi^2$ dof test. 

\subsection{A 1D example}\label{subsec:1D}
For the first numerical example, we will consider a 1D blurring problem.  The matrix $\tilde{\bfA} \in \mathbb{R}^{N \times N}$ is symmetric Toeplitz where the first row is 
\begin{align}    
\bfz = \frac{1}{\sqrt{2\pi\sigma^2}}\begin{bmatrix}
        \text{exp}\left(\frac{-[0:\text{band}-1]^2}{2\sigma^2}\right) & {\bfzero_{N-\text{band}}}
    \end{bmatrix}. \label{eq:zrow1}
\end{align}
We use $N=512$, $\sigma^2 = 24$, and $\text{band} = 60$.  Gaussian noise with $\text{SNR} = 20$ is added to the blurred signal to produce $\tilde{\bfb}$. To demonstrate the capabilities of our methods, we use a signal with sharp edges. The true signal $\bfx_{true}$ and the blurred and noisy data $\tilde{\bfb}$ are given in \cref{fig:xTrue}. 
  We solve this problem with SB and MM
  for $\bfL \in \mathbb{R}^{(N-1) \times N}$ defined as the discretization of the first derivative operator with zero boundary conditions:
\begin{align*}
    \bfL = \begin{bmatrix}
        -1 & 1 & 0 &\cdots && 0\\
        0 & -1 & 1 & 0 & \cdots& 0 \\
        \vdots & & \ddots & \ddots && \vdots \\
        0 & \cdots & 0 & -1 & 1 & 0 \\
        0 & \cdots & & 0& -1 & 1
    \end{bmatrix}. \label{eq:L1}
\end{align*}
Given the size of this 1D example, we use the GSVD of $\{\bfA,\bfL\}$ to solve the inner problem in both SB and MM. Four different methods for selecting $\lambda$ at each iteration are tested: GCV, the central $\chi^2$ dof test, the non-central $\chi^2$ dof test, and DP.  These are compared with the results obtained using the \textit{optimal} fixed $\lambda$.  In the $\chi^2$ dof tests, the significance level of $\alpha = 0.999$  corresponds to a bound of $0.042$ on $F(\lambda)$. For $F_C(\lambda)$, the bound will be at least $0.042$ but it increases with $\lambda$ and $\norm{\bfx^{(k)}-\bfx_0}_2$.  In this example, the bound on $F_C(\lambda)$ ranges from $0.042$ to $0.1$, with the larger value obtained in the initial iterations. For the four selection methods, $\text{TOL}_\lambda = 0.01$ is applied and compared with selecting parameters at every iteration. 

\subsubsection{Split Bregman}

\begin{figure}
    \centering
    \subfigure[1D deblurring problem: $\bfx$ and $\tilde{\bfb}$ \label{fig:xTrue}]{\includegraphics[width=3.65cm]{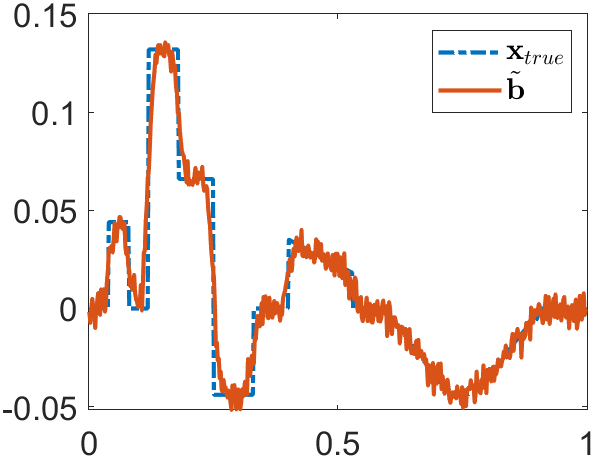}}
    \subfigure[SB, Optimal $\lambda$, $\lambda=232.6$
\label{fig:SBna:a}]{
\includegraphics[width=3.65cm]{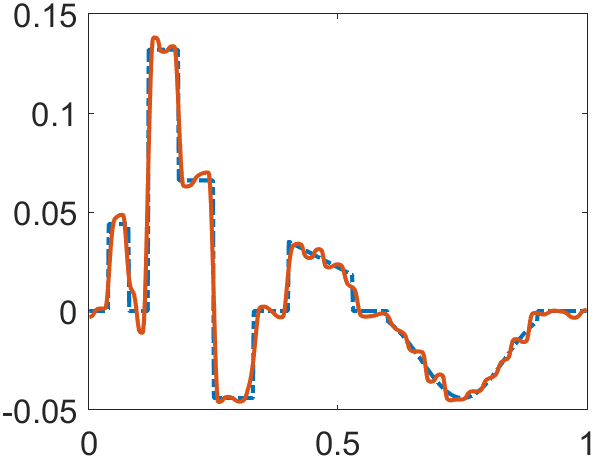}}
\subfigure[MM, Optimal $\lambda$, $\lambda=464.2$
\label{fig:MMna:a}]{
\includegraphics[width=3.65cm]{Images/1Dexample/PS11MM_Opt.png}}
    \caption{ \cref{fig:xTrue}: $\bfx$ and $\tilde{\bfb}$ for the 1D deblurring problem ($\text{SNR} = 20$). \cref{fig:SBna:a,fig:MMna:a}show the SB and MM solutions where $\lambda$ is fixed at the optimal.} 
    \label{fig:Setup}
\end{figure}

\begin{figure}
    \centering
    \subfigure[Relative error\label{fig:SBn:a}]{
\includegraphics[scale=.28]{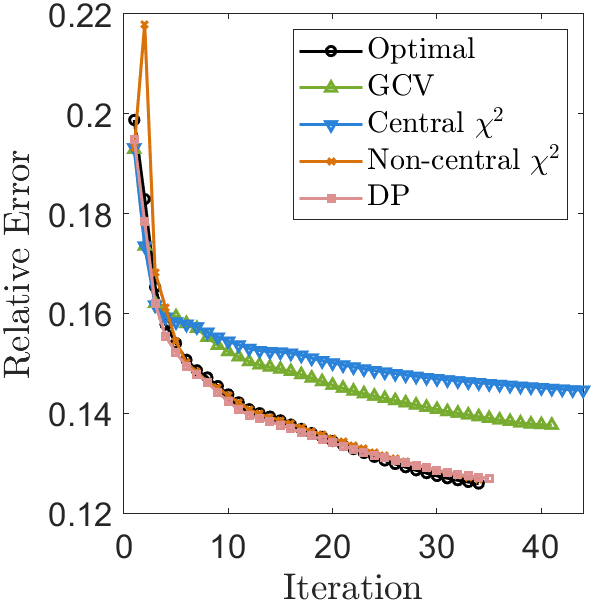}}
\subfigure[Relative change in $\bfx$ \label{fig:SBn:b}]{
\includegraphics[scale=.28]{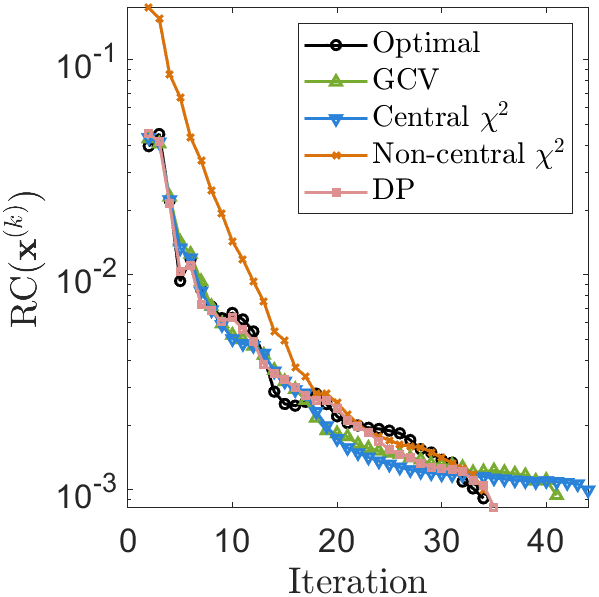}}
\subfigure[ Selection of $\lambda$ \label{fig:SBn:c}]{
\includegraphics[scale=.28]{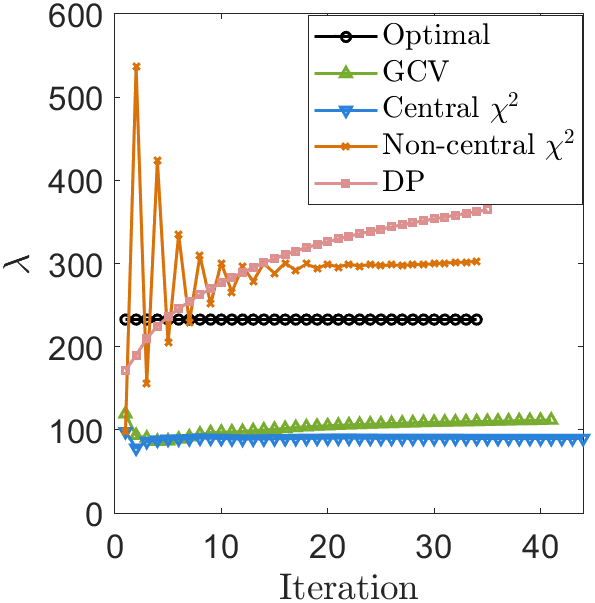}}
\subfigure[ Relative change in $\lambda^2$ \label{fig:SBn:d}]{
\includegraphics[scale=.28]{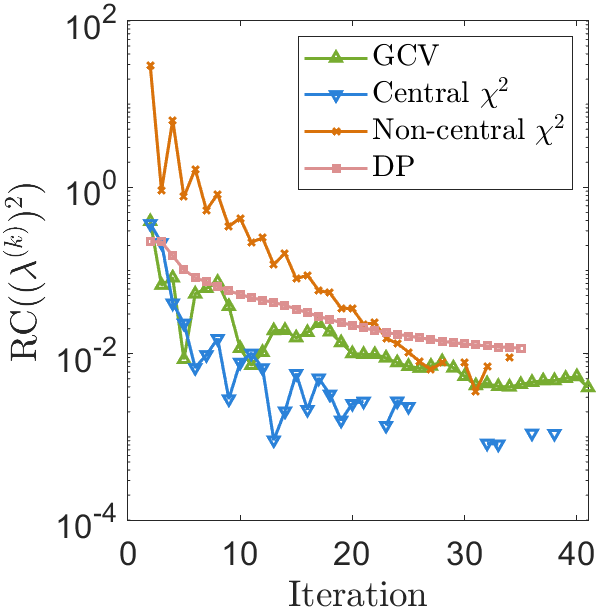}}
\subfigure[ISNR \label{fig:SBn:e}]{
\includegraphics[scale=.28]{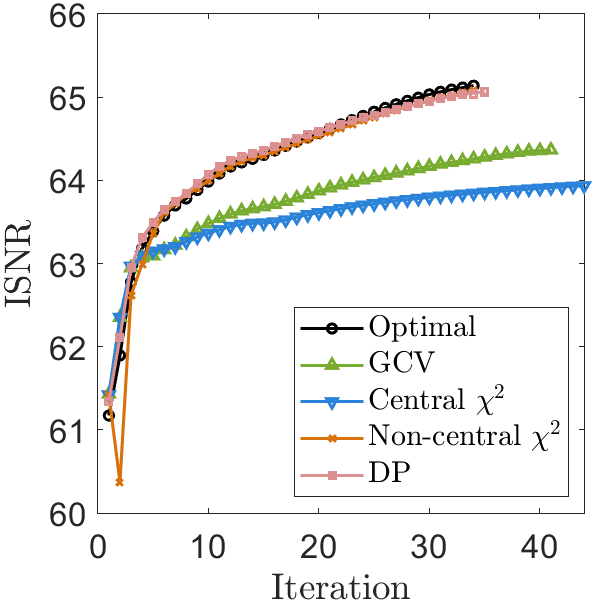}}
    \caption{Results for SB applied to \cref{fig:xTrue} where $\lambda$ is fixed at the optimal $\lambda=232.6$, or selected at each iteration with GCV or the $\chi^2$ dof tests.
    \Cref{fig:SBn:a} plots the $\text{RE}$ by iteration, \cref{fig:SBn:b} plots the relative change in $\bfx$, \cref{fig:SBn:c} plots the $\lambda$ selected, \cref{fig:SBn:d} plots the relative change in $\lambda^2$, and \cref{fig:SBn:e} plots the ISNR.}  
    \label{fig:SBn}
\end{figure}

\begin{figure}
    \centering
\subfigure[$\lambda$ selected by GCV
\label{fig:SBna:b}]{
\includegraphics[width=3.65cm]{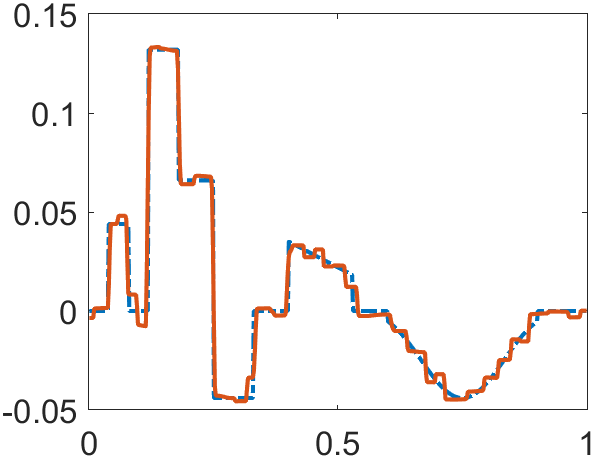}}
    \subfigure[ $\lambda$ selected by central $\chi^2$
    \label{fig:SBna:c}]{
\includegraphics[width=3.65cm]{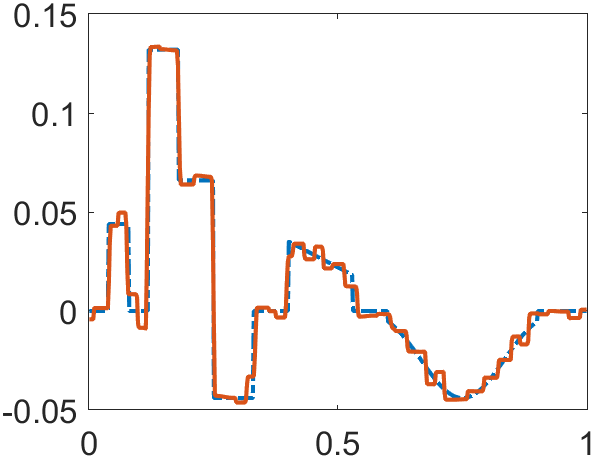}}
    \subfigure[ $\lambda$ selected by non-central $\chi^2$
    \label{fig:SBna:d}]{
\includegraphics[width=3.65cm]{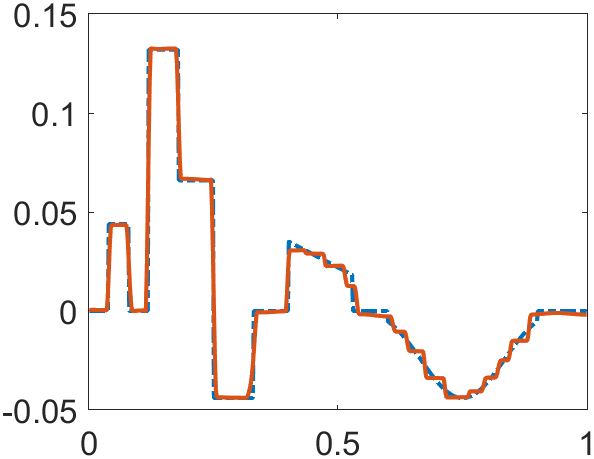}}
    \subfigure[ $\lambda$ selected by DP
    \label{fig:SBna:e}]{
\includegraphics[width=3.65cm]{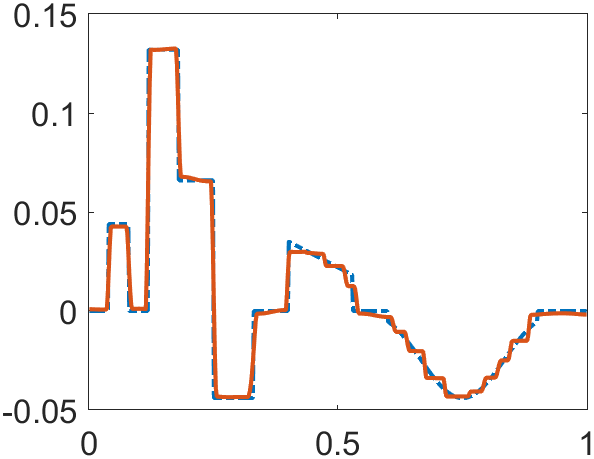}}

\subfigure[$\lambda$ selected by GCV, $\text{TOL}_\lambda = 0.01$
\label{fig:SBna:f}]{
\includegraphics[width=3.65cm]{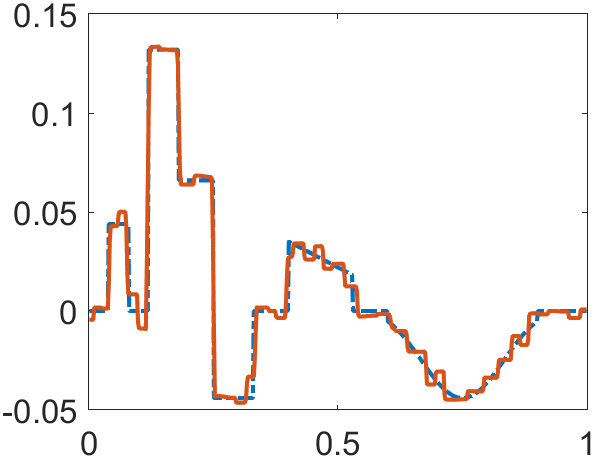}}
\subfigure[$\lambda$ selected by central $\chi^2$, $\text{TOL}_\lambda = 0.01$ 
\label{fig:SBna:g}]{
\includegraphics[width=3.65cm]{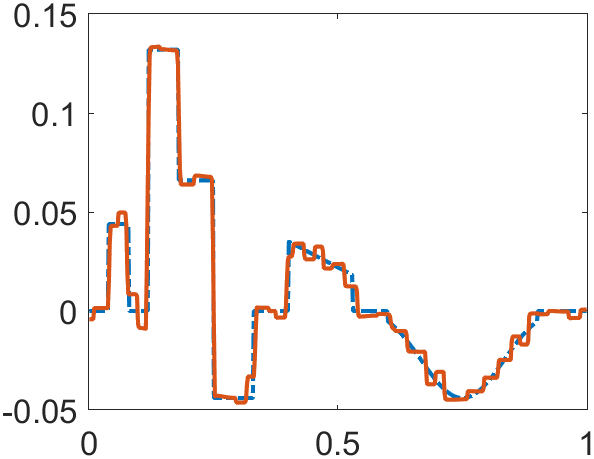}}
\subfigure[$\lambda$ selected by non-central $\chi^2$, $\text{TOL}_\lambda = 0.01$ \label{fig:SBna:h}]{
\includegraphics[width=3.65cm]{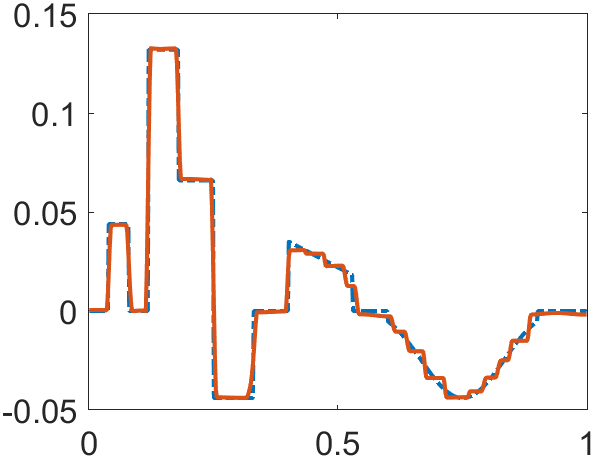}}
    \subfigure[ $\lambda$ selected by DP, $\text{TOL}_\lambda = 0.01$
    \label{fig:SBna:i}]{
\includegraphics[width=3.65cm]{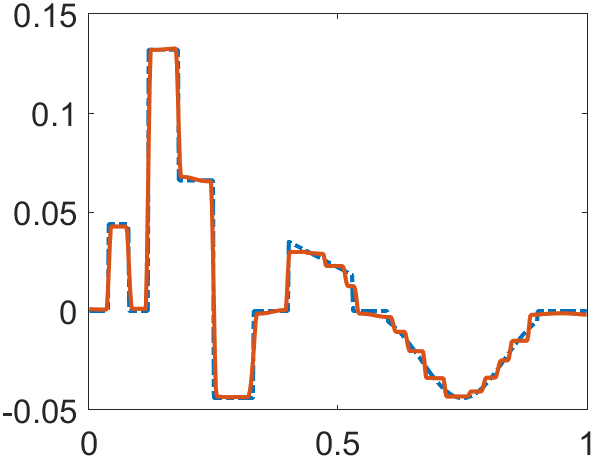}}

    \caption{SB solutions at convergence for \cref{fig:xTrue}.}
    \label{fig:SBna}
\end{figure}

In SB, we set $\tau = 0.005$ as $\text{RC}(\bfx^{(k)})$ decreases smoothly until convergence for this $\tau$.
The ideal value of $\tau$ depends on the noise level, where $\tau$ should be smaller when the noise level is lower. In this case, a larger value of $\tau$ causes $\text{RC}(\bfx^{(k)})$ to spike upward at different iterations, even for a fixed value of $\lambda$.
The \textit{optimal} fixed $\lambda$ in this case is $\lambda = 232.6$  and the corresponding solution is shown in \cref{fig:SBna:a}. 
The results for SB with these selection methods are shown in \cref{fig:SBn}, with solutions provided in \cref{fig:SBna}.  The values of $\lambda$ selected by the non-central $\chi^2$ dof test oscillate in the first 15 iterations.
This is to be anticipated with the choice $\bfx^{(k)}$ for $\overline{\bfx}$. Initially, we expect that $\bfx^{(k)}$ does not serve as a good estimate for $\overline{\bfx}$ to be used at step $k+1$ because, in the beginning steps, the values for $\bfx^{(k)}$ are far from convergence. The bound on $F_C(\lambda)$ oscillates with $\lambda$ and is largest at iteration $2$ where it is $0.1$. This suggests that a smaller $\alpha$ might be better in the early iterations when $\bfx^{(k)}$ is further from $\overline{\bfx}$ as this would increase the bound and therefore dampen the oscillations.  Despite the oscillations for the early iterations, the selection of $\lambda$ still converges for $\alpha = 0.999$. 
On the other hand, the values of $\lambda$ selected by GCV and the central $\chi^2$ dof test converge earlier.  

\subsubsection{Majorization-Minimization}

\begin{figure}
    \centering
    \subfigure[Relative error\label{fig:MMn:a}]{
\includegraphics[scale=.28]{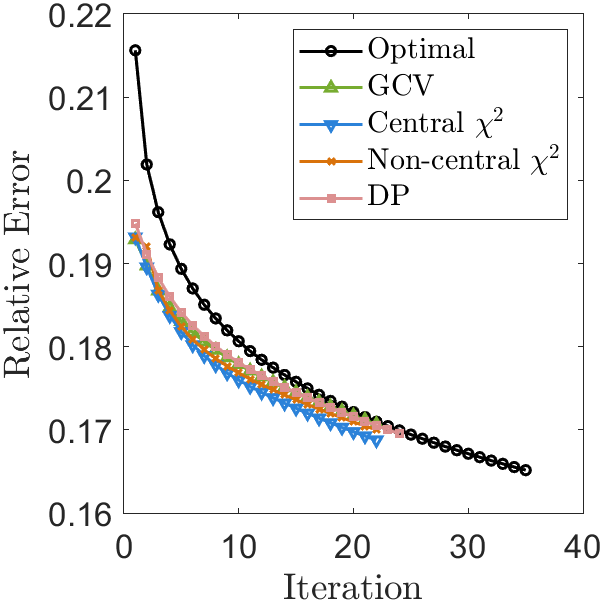}}
\subfigure[Relative change in $\bfx$ \label{fig:MMn:b}]{
\includegraphics[scale=.28]{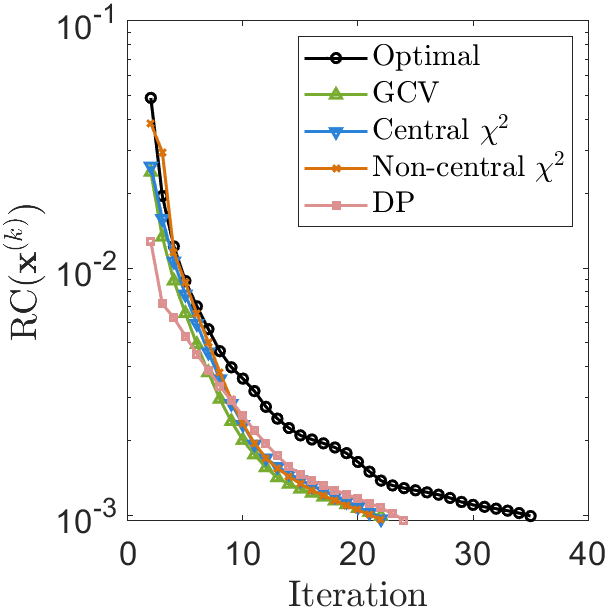}}
\subfigure[ Selection of $\lambda$ \label{fig:MMn:c}]{
\includegraphics[scale=.28]{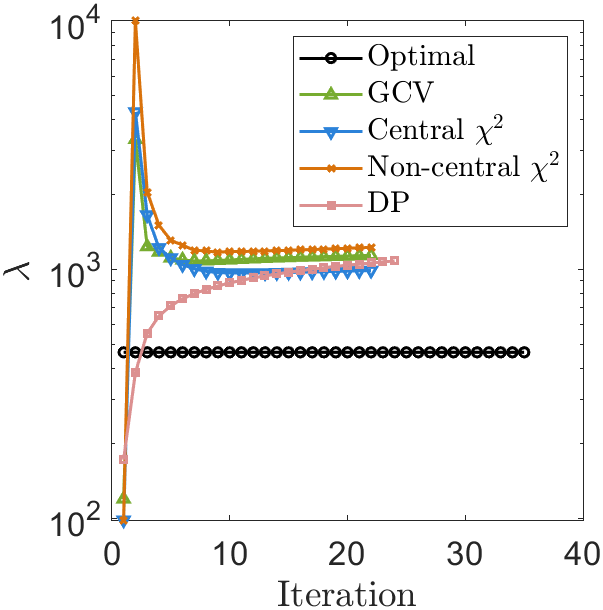}}
\subfigure[ Relative change in $\lambda^2$ \label{fig:MMn:d}]{
\includegraphics[scale=.28]{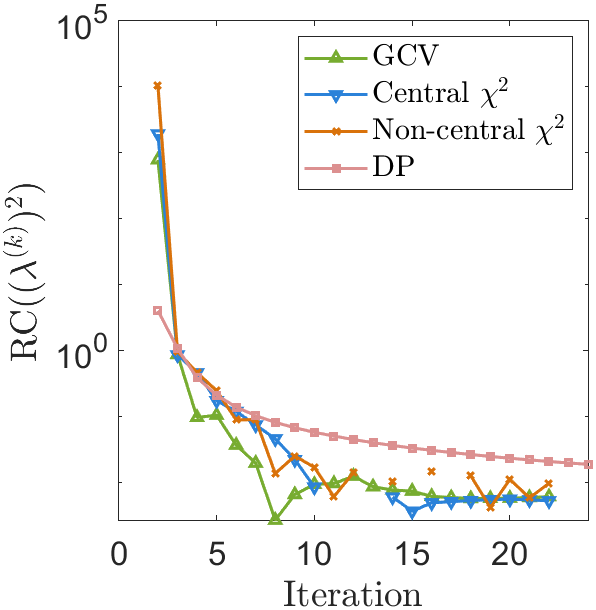}}
\subfigure[ISNR \label{fig:MMn:e}]{
\includegraphics[scale=.28]{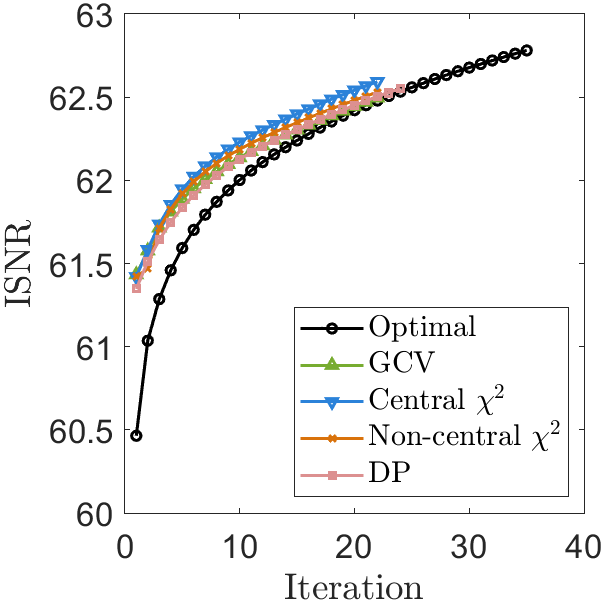}}
    \caption{Results for MM applied to \cref{fig:xTrue} where $\lambda$ is fixed at the optimal $\lambda=464.2$, or selected at each iteration with GCV or the $\chi^2$ dof tests.
    \Cref{fig:MMn:a} plots the $\text{RE}$ by iteration, \cref{fig:MMn:b} plots the relative change in $\bfx$, \cref{fig:MMn:c} plots the $\lambda$ selected, \cref{fig:MMn:d} plots the relative change in $\lambda^2$, and \cref{fig:MMn:e} plots the ISNR. } 
    \label{fig:MMn}
\end{figure}

\begin{figure}
    \centering
\subfigure[$\lambda$ selected by GCV
\label{fig:MMna:b}]{
\includegraphics[width=3.65cm]{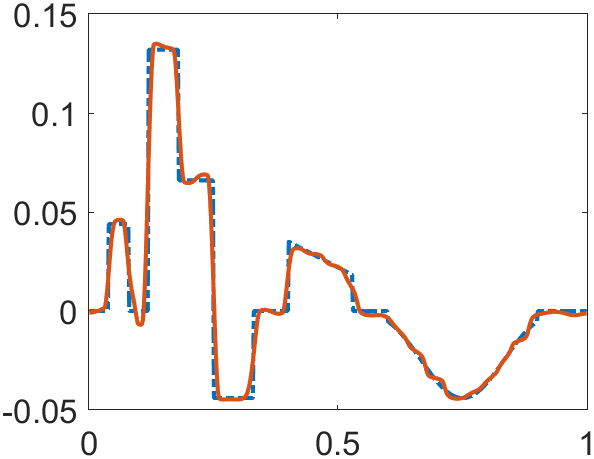}}
    \subfigure[ $\lambda$ selected by central $\chi^2$
    \label{fig:MMna:c}]{
\includegraphics[width=3.65cm]{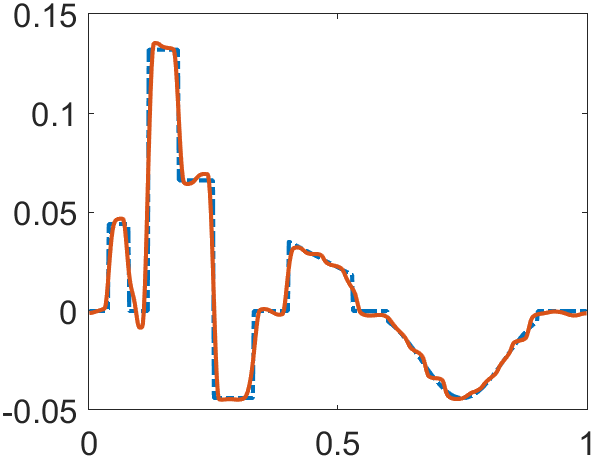}}
    \subfigure[ $\lambda$ selected by non-central $\chi^2$
    \label{fig:MMna:d}]{
\includegraphics[width=3.65cm]{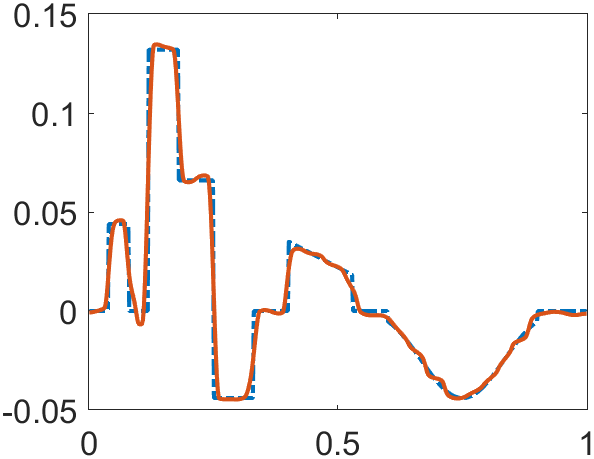}}
    \subfigure[ $\lambda$ selected by DP
    \label{fig:MMna:e}]{
\includegraphics[width=3.65cm]{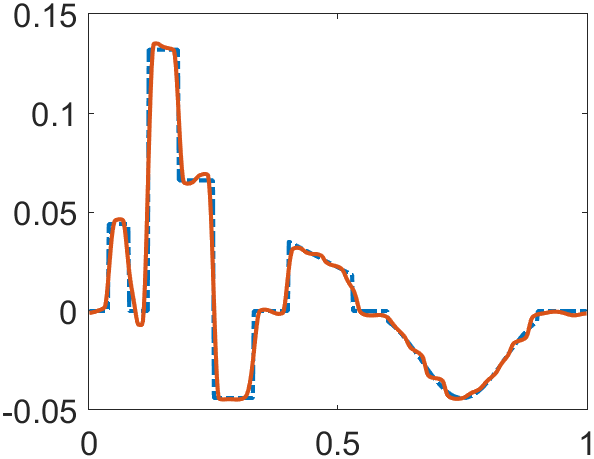}}

\subfigure[$\lambda$ selected by GCV, $\text{TOL}_\lambda = 0.01$
\label{fig:MMna:f}]{
\includegraphics[width=3.65cm]{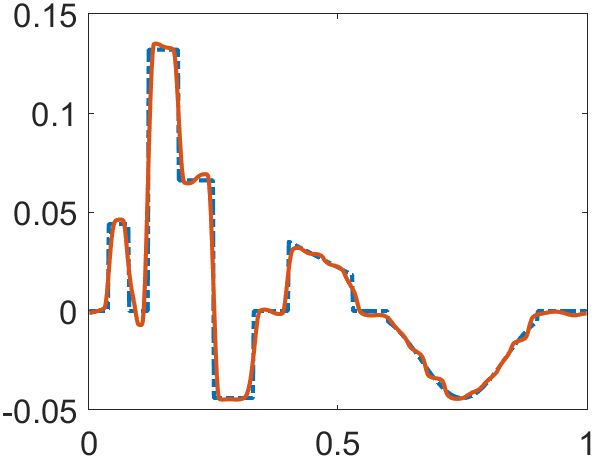}}
\subfigure[$\lambda$ selected by central $\chi^2$, $\text{TOL}_\lambda = 0.01$ 
\label{fig:MMna:g}]{
\includegraphics[width=3.65cm]{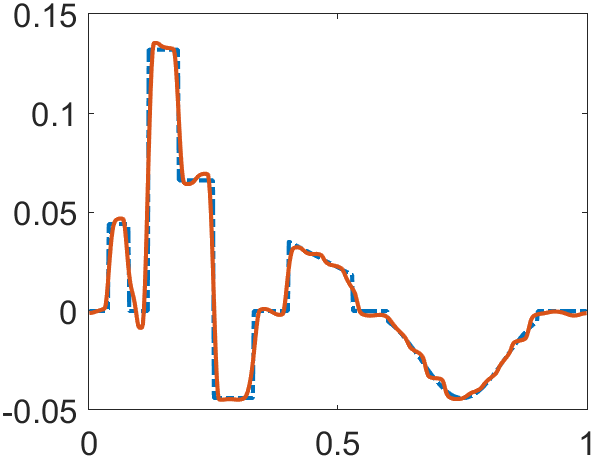}}
\subfigure[$\lambda$ selected by non-central $\chi^2$, $\text{TOL}_\lambda = 0.01$ \label{fig:MMna:h}]{
\includegraphics[width=3.65cm]{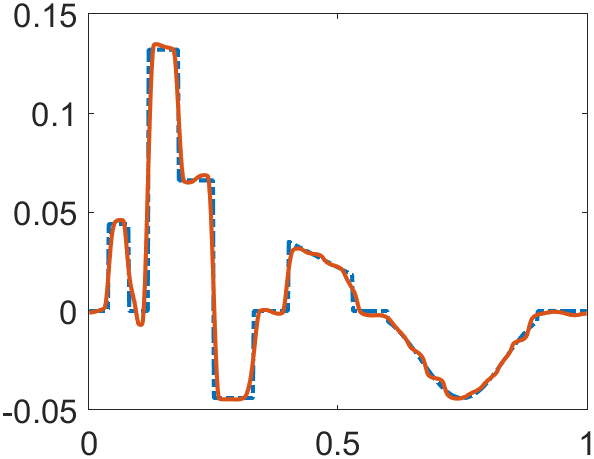}}
    \subfigure[ $\lambda$ selected by DP, $\text{TOL}_\lambda = 0.01$
    \label{fig:MMna:i}]{
\includegraphics[width=3.65cm]{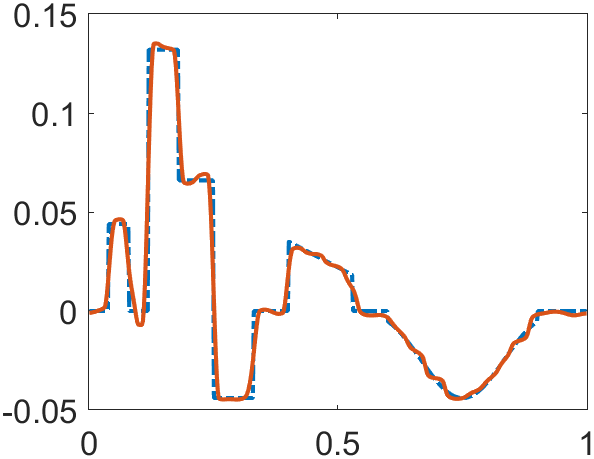}}

    \caption{MM solutions at convergence for \cref{fig:xTrue}.}
    \label{fig:MMna}
\end{figure}

In MM, we fix $\varepsilon = 0.0003$, which is near the suggested value in \cite{buccini2021choice,buccini2019l} relative to the magnitude of $\bfx$. The ideal value of $\varepsilon$ depends on the magnitude of $\bfu^{(k)} = \bfL\bfx^{(k)}$.  
In this case, the \textit{optimal} fixed $\lambda$ is $\lambda= 464.2$, which produces the solution in \cref{fig:MMna:a}.
The results of the selection methods are shown in \cref{fig:MMn} with solutions in \cref{fig:MMna}. From these plots, both GCV and the $\chi^2$ dof tests select $\lambda$ well, leading to better convergence than the \textit{optimal}. 
Initial values of $\lambda$ found using the non-central $\chi^2$ dof test oscillate, but still converge.  The bound on $F_C(\lambda)$ changes with the oscillations, reaching a maximum of $0.068$ in iteration $2$.  These initial steps serve to find a stabilizing value for $\lambda$ that is suitable for the noise in the  steps after the tolerance $\text{TOL}_\lambda$ has been achieved. 

\subsubsection{Discussion on the one-dimensional results}\label{sec:onedres}
The results using the SB and MM iterative methods with the regularization parameter methods are summarized in \cref{tab:Ex1}.  In SB, the non-central $\chi^2$ dof test and DP perform best of the parameter selection methods, converging in fewer iterations and to solutions with smaller RE and larger ISNR. When $\text{TOL}_\lambda = 0.01$, GCV and the central $\chi^2$ dof test converge faster than the other methods but have solutions with a larger $\text{RE}$.  In MM, the four selection methods each converge to solutions having approximately the same $\text{RE}$. Of these methods, GCV is the fastest. 
In general, SB and MM are different, with SB taking longer to converge but to solutions with a smaller $\text{RE}$ and larger \text{ISNR}.
The results suggest that the best method with respect to $\text{RE}$ and ISNR is SB with the non-central $\chi^2$ dof test. Using $\text{TOL}_\lambda = 0.01$ does not significantly impact the results of the selection methods. 

\begin{table}[htp]
\centering
\caption{$\text{RE}$, $\text{ISNR}$, iterations, and computation time (in seconds) for the solutions to the 1D example in \cref{fig:xTrue}. The best results are shown in boldface, excluding the methods where $\lambda$ is fixed at the \textit{optimal}.}
\begin{tabular}{lcccccccc}
\hline
                         & \multicolumn{4}{c}{No $\text{TOL}_\lambda$} & \multicolumn{4}{c}{$\text{TOL}_\lambda=0.01$} \\ \cmidrule(lr){2-5} 
                         \cmidrule(lr){6-9} 
Method                   & RE            & ISNR         & Iter. & Time        & RE            & ISNR          & Iter. & Time        \\ \hline
SB, Optimal             & 0.126         & 65.14        & 34  & 0.04        &               &               &               \\
SB, GCV                  & 0.138         & 64.37        & 41  & 0.34         & 0.146         & 63.86         & 46  & 0.06          \\
SB, Central $\chi^2$     & 0.145         & 63.93        & 44  & 0.10         & 0.145         & 63.92         & 46  & 0.08          \\
SB, Non-central $\chi^2$ & \textbf{0.127}         & \textbf{65.08}        & 34 & 0.20          & \textbf{0.127}         & \textbf{65.08}         & 34  & 0.13          \\
SB, DP                   & \textbf{0.127}         & 65.06        & 35 & 0.18          & 0.127         & 65.06         & 35  & 0.17          \\ \hline
MM, Optimal             & 0.165         & 62.78        & 35  & 0.04         &               &               &               \\
MM, GCV                  & 0.171         & 62.50        & \textbf{22} & 0.06          & 0.170         & 62.53         & 23  & \textbf{0.04}          \\
MM, Central $\chi^2$     & 0.169         & 62.59        & \textbf{22} & 0.07          & 0.169         & 62.59         & \textbf{22}  & 0.06          \\
MM, Non-central $\chi^2$ & 0.170         & 62.53        & \textbf{22}  & 0.11         & 0.170         & 62.54         & \textbf{22}  & 0.07          \\
MM, DP                   & 0.170         & 62.55        & 24 & 0.12          & 0.170         & 62.55         & 24  & 0.12       \\ \bottomrule
\end{tabular}
\label{tab:Ex1}
\end{table}

\subsection{A 2D example} \label{subsec:2D}
For our 2D example, we use an image deblurring problem.  We use the cameraman image in \cref{fig:Ex2Setup:a}, 
which is $512$ by $512$ pixels.  Since the results are indistinguishable when using full images, we will instead focus on the zoomed-in section of $\bfx_{true}$ shown in \cref{fig:Ex2Setup:xt}. The matrix $\tilde{\bfA} \in \mathbb{R}^{512^2 \times 512^2}$ now models a separable blur, defined by $\tilde{\bfA} = (\bfC_\bfz \otimes \bfC_\bfz)$, where $\otimes$ defines a Kronecker product and the matrix $\bfC_\bfz$ is a circulant matrix with first row given by \cref{eq:zrow1} with $N=512$, $\text{band} = 40$, and $\sigma^2 = 16$.  White Gaussian noise with $\text{SNR} = 20$ is then added to obtain $\tilde{\bfb}$. The blurred and noisy image is given in \cref{fig:Ex2Setup:b}, with the blurred zoom-in shown in \cref{fig:Ex2Setup:d}. We set $\bfL = \bfD_{2D}$, where
\begin{align}
        \bfD_{2D} = \begin{bmatrix}
        \bfI \otimes \bfD_{1D} \\
        \bfD_{1D} \otimes \bfI
    \end{bmatrix}\nonumber
\end{align}
and $\bfD_{1D} \in \mathbb{R}^{N \times N}$ is the discretization of the first derivative with periodic boundary conditions:
\begin{align}
    \bfD_{1D} = \begin{bmatrix}
        -1 & 1 & 0 &0&\cdots &0\\
        0 & -1 & 1 & 0 & \cdots& 0 \\
        \vdots & & \ddots & \ddots && \vdots \\
        0 & \cdots & 0 & -1 & 1 & 0 \\
        0 &0& \cdots & 0& -1 & 1 \\
        1 & 0& \cdots & 0 &0 & -1
    \end{bmatrix}. 
\nonumber
\end{align}
To solve this problem, we use the weighted $\bfA$ and $\bfb$ as given in \cref{eq:Reweighting}.  In this case, $p=524,288$ and  $n=262,144$ so $p = 2n > n$. 
Instead of the GSVD, for this problem we use the discrete Fourier transform \cite{hansen2006deblurring} for which we have the mutual decomposition
\begin{align*}
    \bfA &= \boldsymbol{\mathcal{F}}^* \boldsymbol\Lambda_{\bfA} \boldsymbol{\mathcal{F}} \\
    \bfL &= \begin{bmatrix}
        \bfI_N \otimes \bfD_{1D} \\
        \bfD_{1D} \otimes \bfI_N
    \end{bmatrix} = \begin{bmatrix}
        \boldsymbol{\mathcal{F}}^* & {\bfzero} \\
        {\bfzero} & \boldsymbol{\mathcal{F}}^*
    \end{bmatrix} \begin{bmatrix}
        \bfC \\
        \bfD
    \end{bmatrix}
    \boldsymbol{\mathcal{F}}.
\end{align*}
Here $\boldsymbol{\mathcal{F}} \in \mathbb{C}^{n \times n}$ is the 2D discrete Fourier transform matrix, and $\boldsymbol\Lambda_{\bfA}$ is a diagonal matrix with the eigenvalues of $\bfA$. The matrices $\bfC = \diag(c_1,\dots, c_n)$ and $\bfD = \diag(d_1,\dots,d_n)$ are diagonal and defined by $\bfC = \bfI_N \otimes \boldsymbol\Lambda_\bfL$ and $\bfD = \boldsymbol\Lambda_\bfL \otimes \bfI_N$, where
\begin{align*}
    \boldsymbol{\Lambda}_\bfL = \diag(\lambda_1,\dots, \lambda_N), \text{ with } \lambda_j =  \exp{\left(2j\pi \hat{i}/N\right)}-1, \quad j=1,\dots,N.
\end{align*}
The matrix $\bfL^\dagger_{\bfA} = \bfL^\dagger$ is computed as
\begin{align*}
    \bfL_{\bfA}^\dagger = \bfL^\dagger =
        \boldsymbol{\mathcal{F}}^*
    \begin{bmatrix}
    \tilde{\bfC} & \tilde{\bfD}
    \end{bmatrix}
    \begin{bmatrix}
    \boldsymbol{\mathcal{F}} & {\bfzero} \\
        {\bfzero} & \boldsymbol{\mathcal{F}}
        \end{bmatrix},
\end{align*}
where $\tilde{\bfC} = \diag{(\tilde{c}_1,\dots,\tilde{c}_n)}$ and $\tilde{\bfD} = \diag{(\tilde{d}_1,\dots,\tilde{d}_n)}$ are diagonal matrices, with
\begin{align*}
\tilde{c}_i = \begin{cases} \frac{\overline{c_i}}{\left|c_i\right|^2+\left|d_i\right|^2}, &\text{ if} \left|c_i\right|^2+\left|d_i\right|^2 \neq 0 \\ 0, &\text{ if} \left|c_i\right|^2+\left|d_i\right|^2=0 \end{cases}, \text{ and }  \tilde{d}_i = \begin{cases} \frac{\overline{d_i}}{\left|c_i\right|^2+\left|d_i\right|^2}, &\text{ if} \left|c_i\right|^2+\left|d_i\right|^2 \neq 0 \\ 0, &\text{ if} \left|c_i\right|^2+\left|d_i\right|^2=0 \end{cases}.
\end{align*}
In this case,
\begin{align*}
    \bfL\bfL_{\bfA}^\dagger = \begin{bmatrix}
        \boldsymbol{\mathcal{F}}^* & {\bfzero} \\
        {\bfzero} & \boldsymbol{\mathcal{F}}^*
    \end{bmatrix} \begin{bmatrix}
        \bfC \\ \bfD
    \end{bmatrix}
    \begin{bmatrix}
    \tilde{\bfC} & \tilde{\bfD}
    \end{bmatrix}
    \begin{bmatrix}
    \boldsymbol{\mathcal{F}} & {\bfzero} \\
        {\bfzero} & \boldsymbol{\mathcal{F}}
        \end{bmatrix}.
\end{align*}
This matrix is typically not the identity and has four diagonal submatrices of size $n \times n$.

As before, the selection methods used for the 1D case as well as the 2D specific RWP are compared with the \textit{optimal} fixed $\lambda$. 
A confidence level of $\alpha = 0.999$ is used in the $\chi^2$ dof tests, which corresponds to a bound of $0.941$ on $F(\lambda)$ for the central test.  The bound on $F_C(\lambda)$ for the non-central test is larger, ranging from $0.941$ to $1.485$ in this example.

\subsubsection{Split Bregman}
For SB, we set $\tau=0.01$ as this leads to $\text{RC}(\bfx^{(k)})$ decreasing smoothly. 
Other values of $\tau$ were tested, but there was not a significant difference in the results as long as $\tau$ is not extremely large or small. The convergence results given in \cref{fig:SB2D1} and solutions given in \cref{fig:SB2Res}, demonstrate that GCV and the central $\chi^2$ test perform best. Qualitatively, it is hard to distinguish between the solutions obtained using automatic determination of $\lambda$ and the \textit{optimal} solution that is found by searching over a large range of $\lambda$ assuming the known solution. The DP and RWP solutions, however, are not as clear near the handle and the stand.
The $\lambda$ selected by the non-central $\chi^2$ dof test oscillates initially, but the $\text{RE}$ is comparable to the other methods after the third iteration. In terms of RE and ISNR, DP and RWP do not perform as well as the other selection methods.

\begin{figure}
    \centering
    \subfigure[  $\bfx_{true}$\label{fig:Ex2Setup:a}]{
\includegraphics[width=0.18\textwidth]{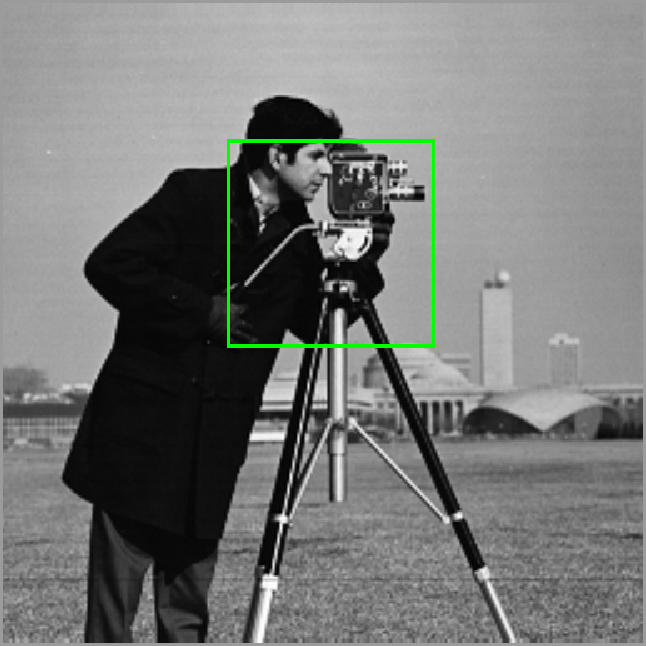}}
    \subfigure[ $\bfx_{true}$ zoom-in \label{fig:Ex2Setup:xt}]{
\includegraphics[width=0.18\textwidth]{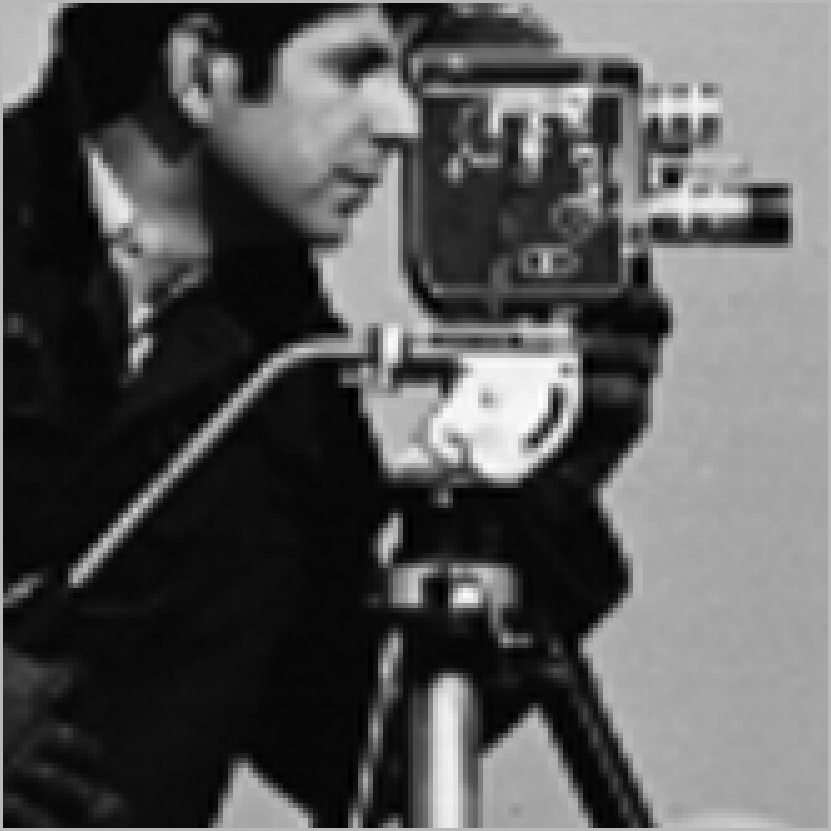}}
    \subfigure[$\tilde{\bfb}$ \label{fig:Ex2Setup:b}]{
\includegraphics[width=0.18\textwidth]{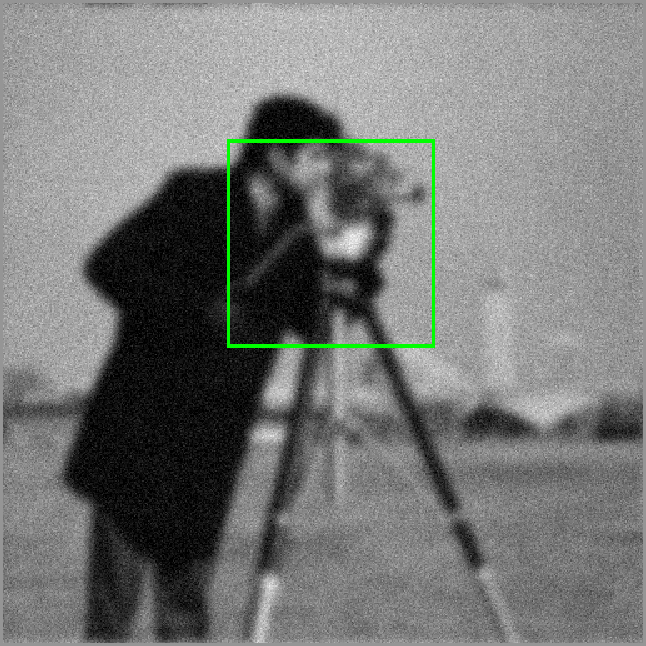}}
    \subfigure[$\tilde{\bfb}$ zoom-in \label{fig:Ex2Setup:d}]{
\includegraphics[width=0.18\textwidth]{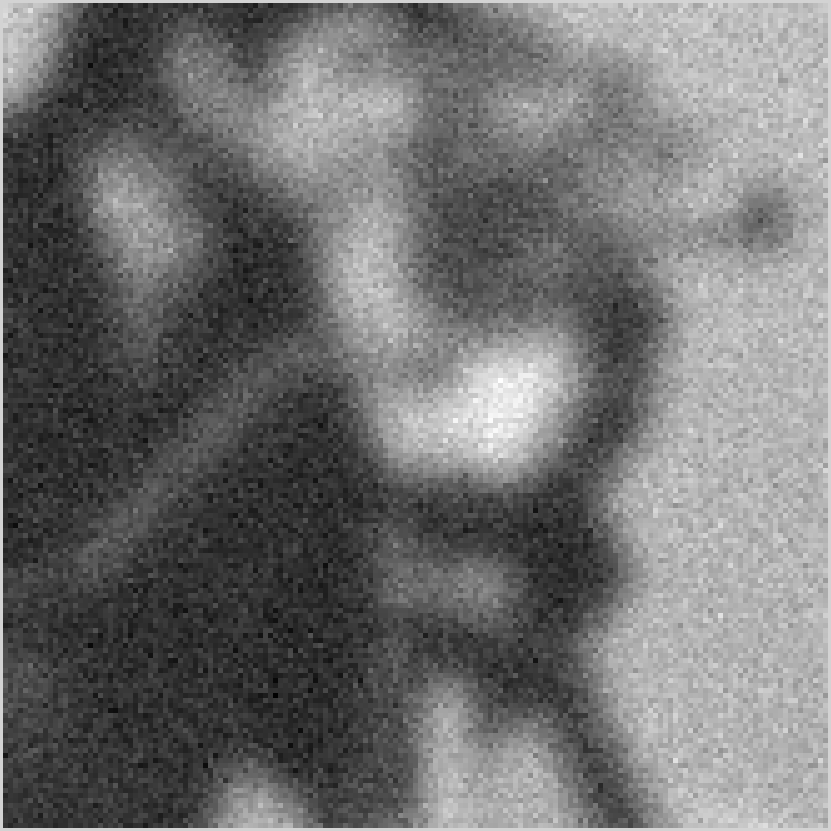}}
    \caption{The true image $\bfx_{true}$ and the blurred and noisy image $\tilde{\bfb}$ for the image deblurring example for image of size $512 \times 512$.  Zoomed-in sections of both images are also shown.}
    \label{fig:Ex2Setup}
\end{figure}

\begin{figure}
    \centering
\subfigure[Relative error \label{fig:SB2D1:a}]{
\includegraphics[scale=.28]{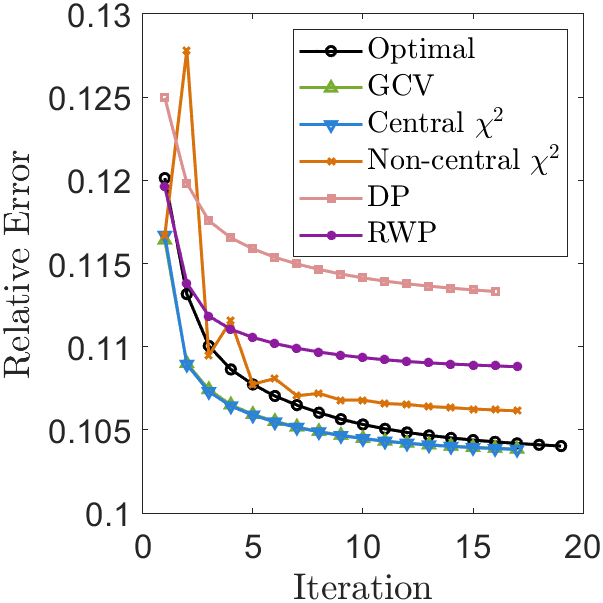}}
\subfigure[Relative change  in $\bfx$\label{fig:SB2D1:b}]{
\includegraphics[scale=.28]{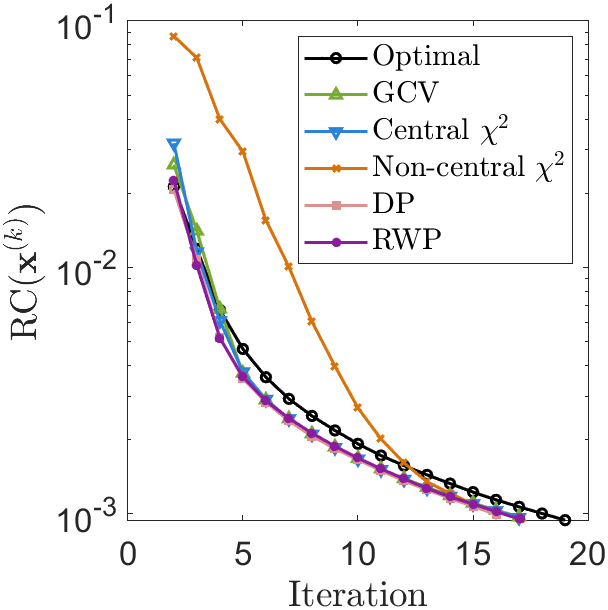}}
\subfigure[ Selection of $\lambda$ \label{fig:SB2D1:c}]{
\includegraphics[scale=.28]{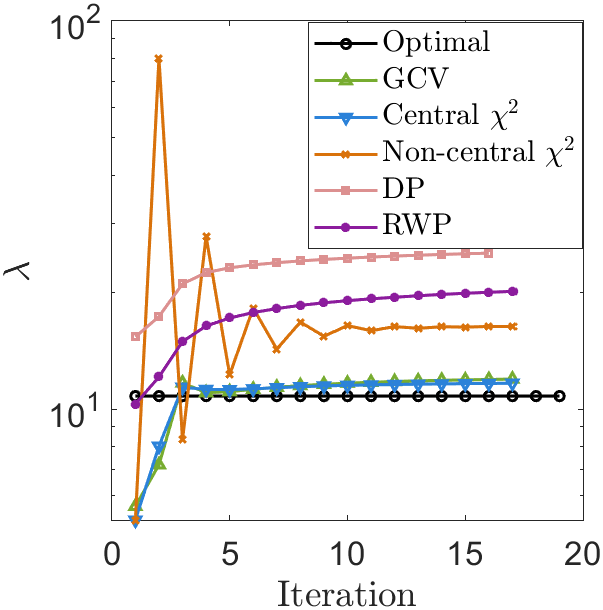}}
\subfigure[Relative change  in $\lambda^2$ \label{fig:SB2D1:d}]{
\includegraphics[scale=.28]{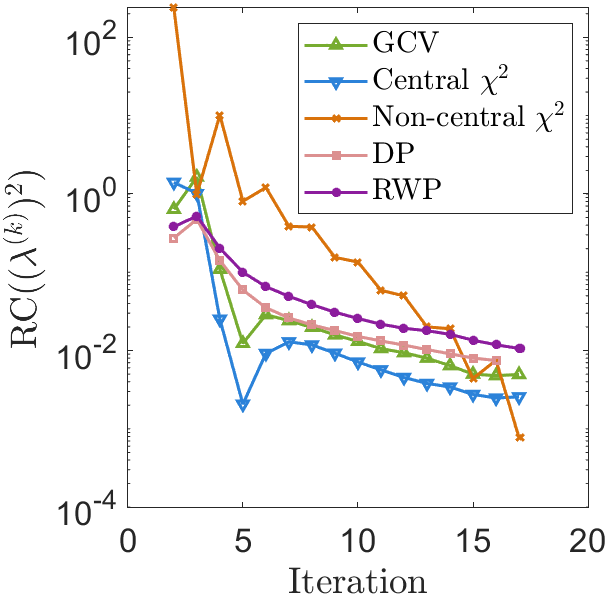}}
   \subfigure[ISNR \label{fig:SB2D1:e}]{
\includegraphics[scale=.28]{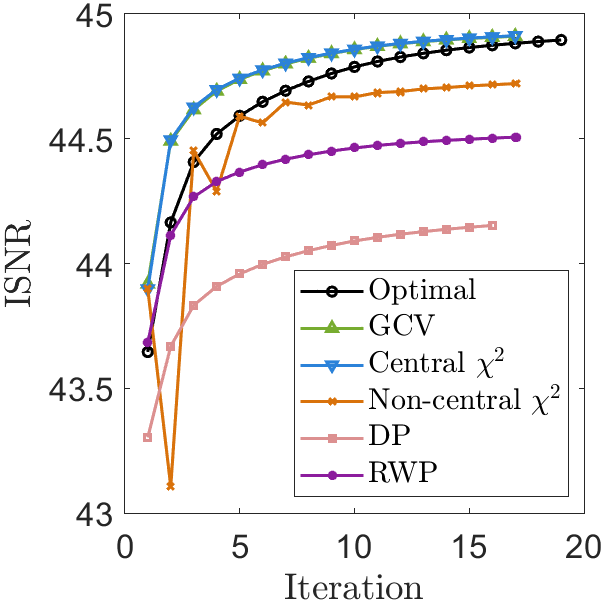}}
    \caption{Results for SB applied to \cref{fig:Ex2Setup:b} where $\lambda$ is fixed at the optimal $\lambda = 10.8$, or selected at each iteration with GCV, the $\chi^2$ dof tests, or DP.
    \Cref{fig:SB2D1:a} plots the $\text{RE}$ by iteration, \cref{fig:SB2D1:b} plots the relative change in $\bfx$, \cref{fig:SB2D1:c} plots the $\lambda$ selected, \cref{fig:SB2D1:d} plots the relative change in $\lambda^2$, and \cref{fig:SB2D1:e} plots the ISNR.}
    \label{fig:SB2D1}
\end{figure}

\begin{figure}
    \centering
        \subfigure[ $\bfx_{true}$ zoom-in \label{fig:Ex2Setupa:xt}]{
\includegraphics[width=0.18\textwidth]{Images/Cam512/PS22SB_XTinsert.png}}
\subfigure[Optimal $\lambda $ ($\lambda = 10.8$) 
\label{fig:SB2Res:a}]{
\includegraphics[width=0.18\textwidth]{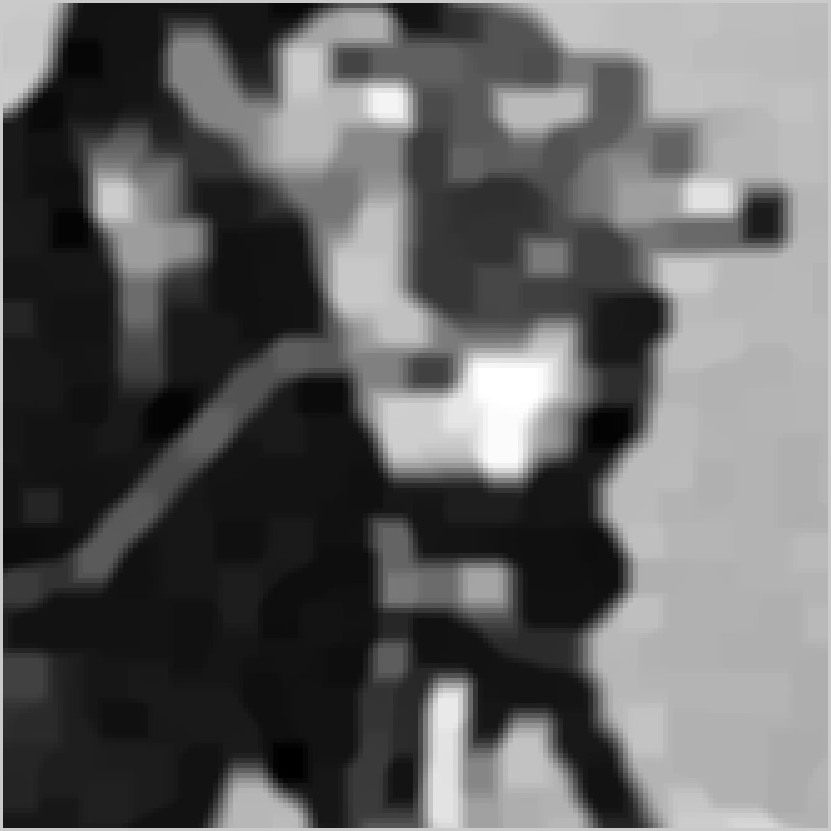}}

    \subfigure[ $\lambda$ selected by GCV
    \label{fig:SB2Res:b}]{
\includegraphics[width=0.18\textwidth]{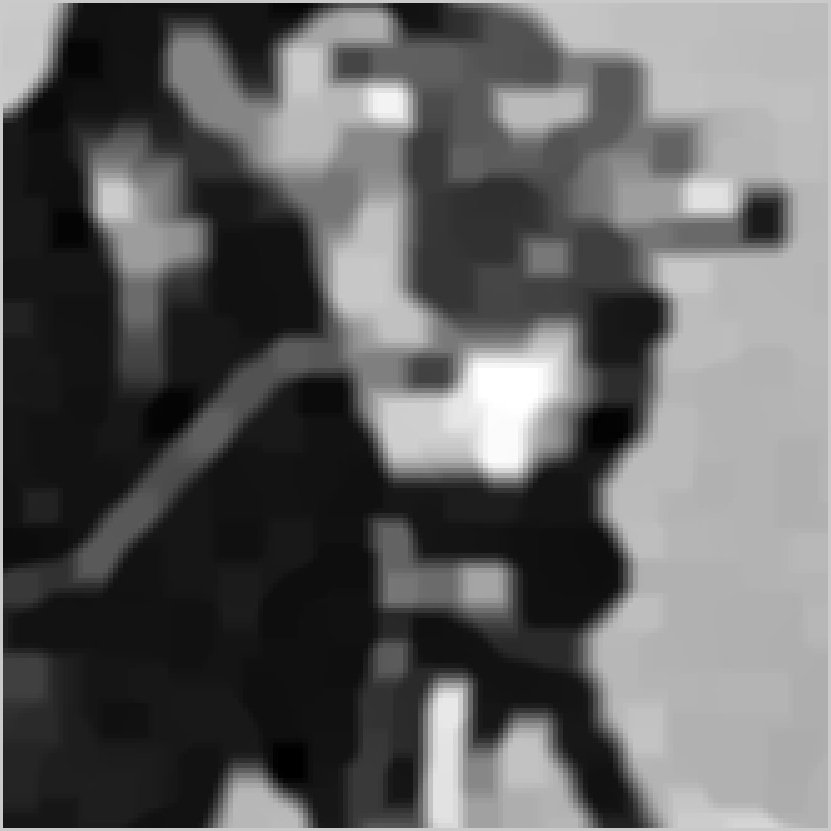}}
\subfigure[$\lambda$ selected by central $\chi^2$
\label{fig:SB2Res:c}]{
\includegraphics[width=0.18\textwidth]{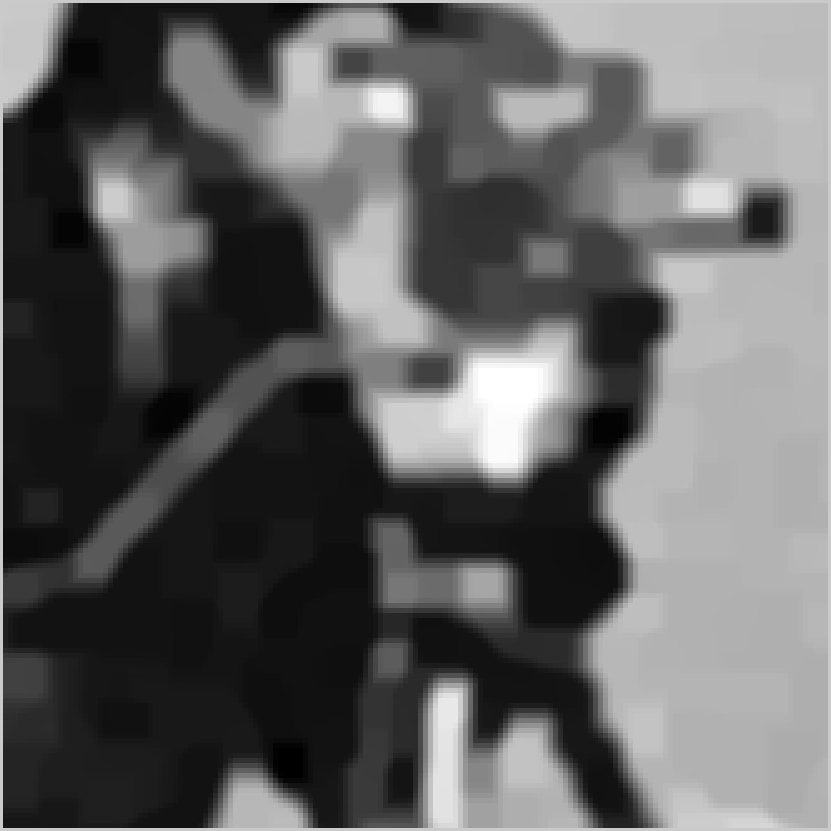}}
\subfigure[$\lambda$ selected by non-central $\chi^2$
\label{fig:SB2Res:d}]{
\includegraphics[width=0.18\textwidth]{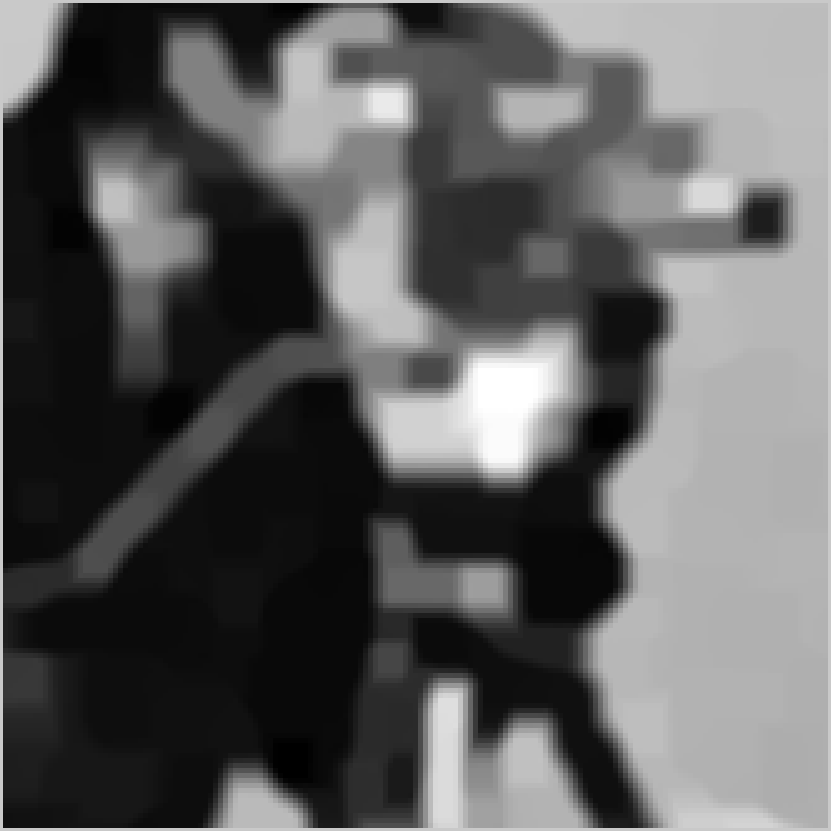}}
\subfigure[$\lambda$ selected by DP
\label{fig:SB2Res:e}]{
\includegraphics[width=0.18\textwidth]{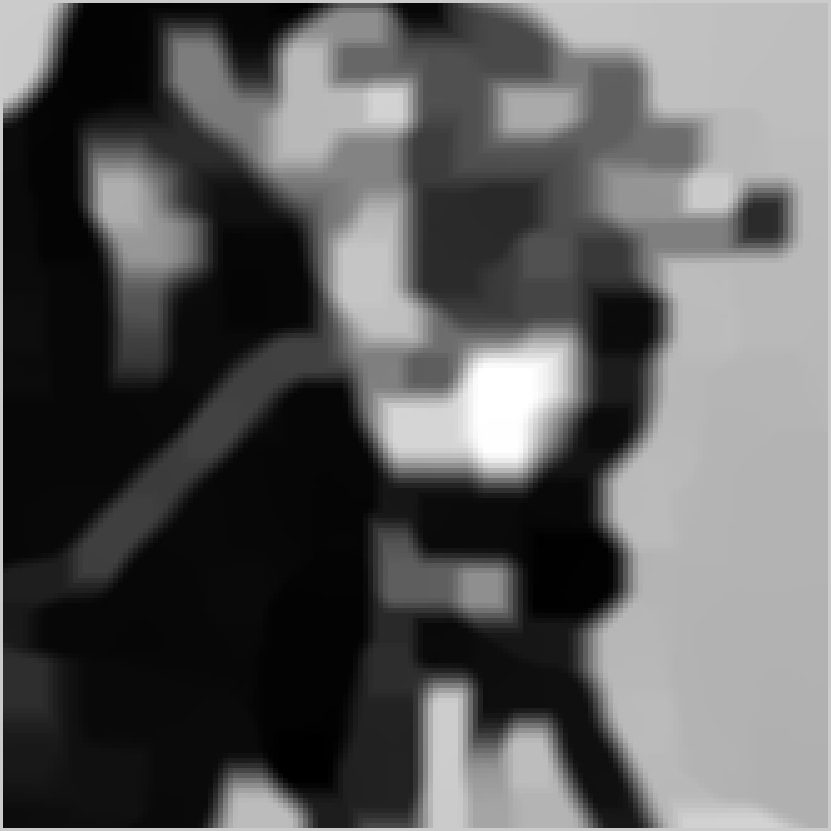}}
\subfigure[$\lambda$ selected by RWP
\label{fig:SB2Res:r}]{
\includegraphics[width=0.18\textwidth]{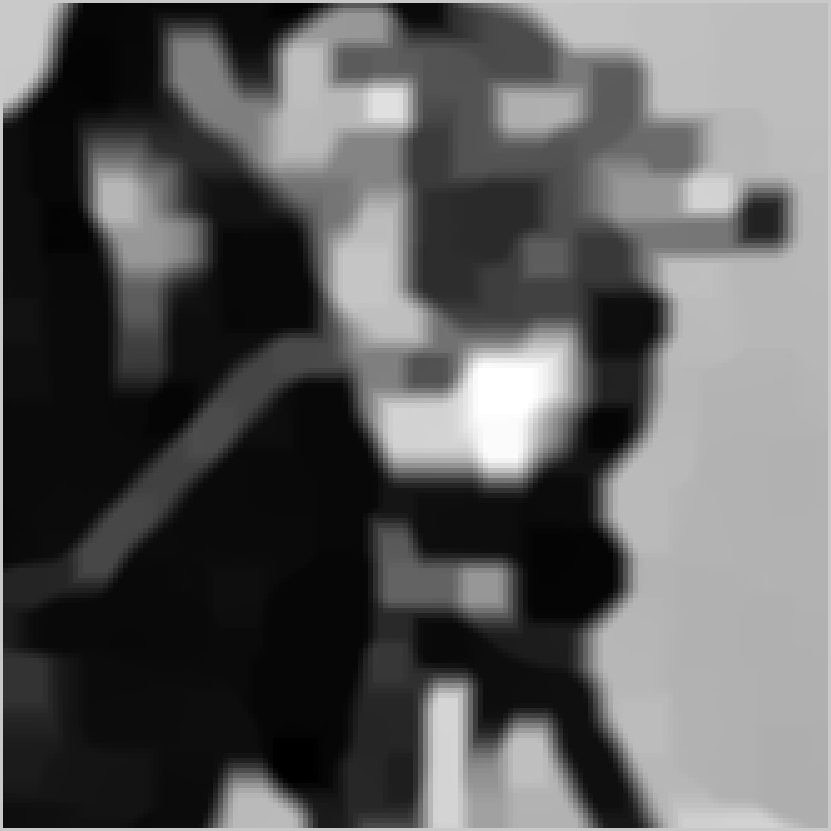}}
\subfigure[$\lambda$ selected by GCV, $\text{TOL}_\lambda=0.01$ \label{fig:SB2Res:f}]{
\includegraphics[width=0.18\textwidth]{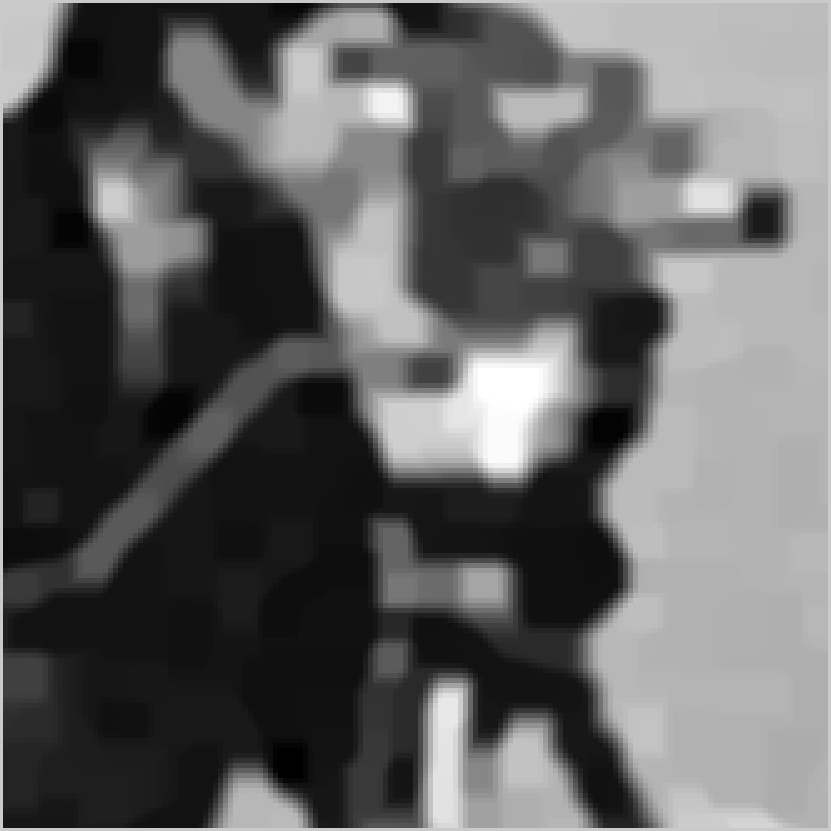}}
\subfigure[ $\lambda$ selected by central $\chi^2$, $\text{TOL}_\lambda=0.01$
    \label{fig:SB2Res:g}]{
\includegraphics[width=0.18\textwidth]{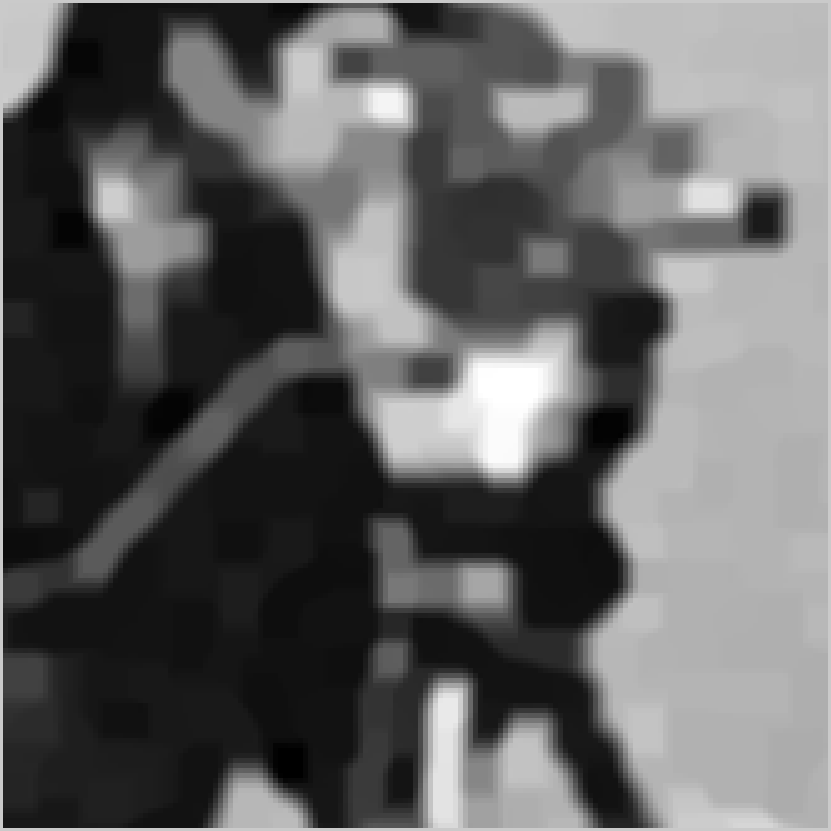}}
\subfigure[$\lambda$ selected by non-central $\chi^2$, $\text{TOL}_\lambda=0.01$ \label{fig:SB2Res:h}]{
\includegraphics[width=0.18\textwidth]{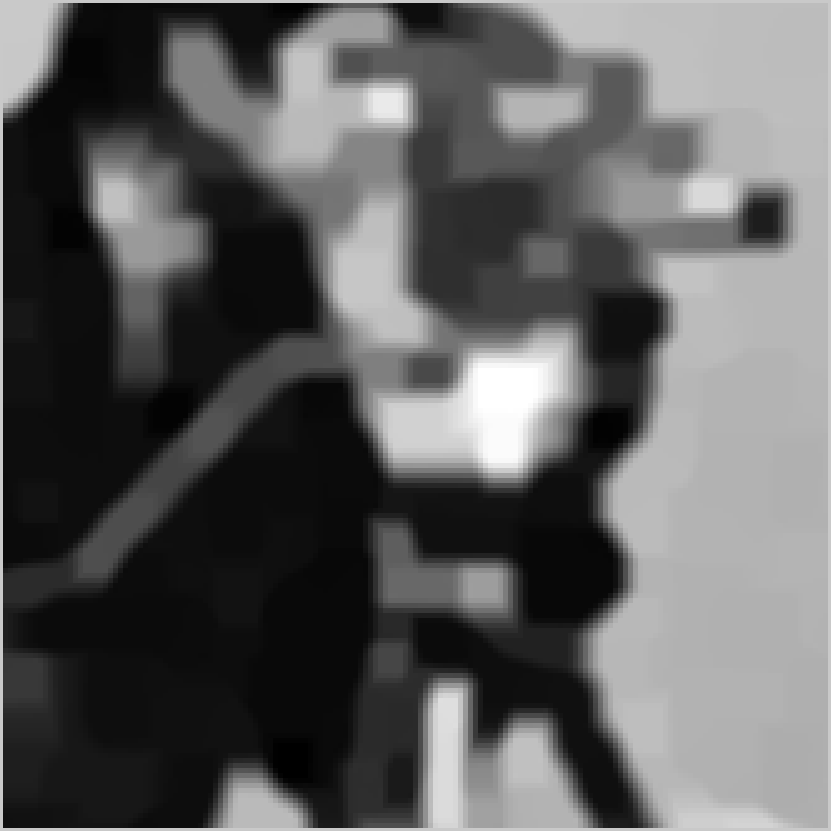}}
\subfigure[$\lambda$ selected by DP, $\text{TOL}_\lambda=0.01$
\label{fig:SB2Res:i}]{
\includegraphics[width=0.18\textwidth]{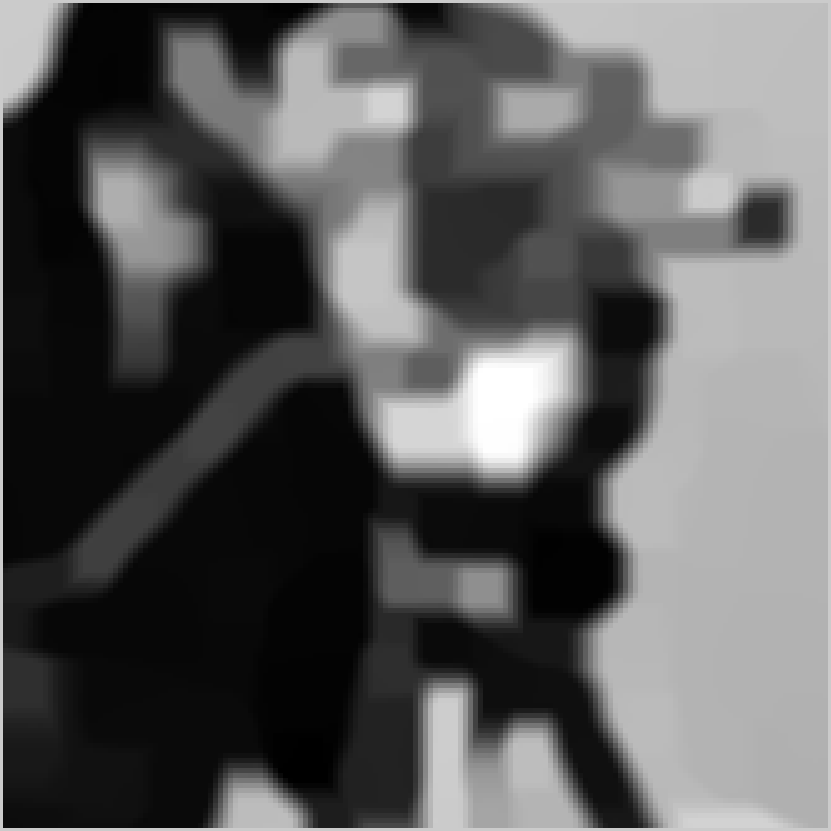}}
\subfigure[$\lambda$ selected by RWP, $\text{TOL}_\lambda=0.01$
\label{fig:SB2Res:rl}]{
\includegraphics[width=0.18\textwidth]{Images/Cam512/PS22SB_DPl.png}}
    \caption{Zoomed-in windows of the SB solutions at convergence for \cref{fig:Ex2Setup:b}. We note that the full images are indistinguishable and are thus not shown.}
    \label{fig:SB2Res}
\end{figure}

\subsubsection{Majorization-Minimization}

\begin{figure}
    \centering
   \subfigure[Relative error \label{fig:MM2D1:a}]{
\includegraphics[scale=.28]{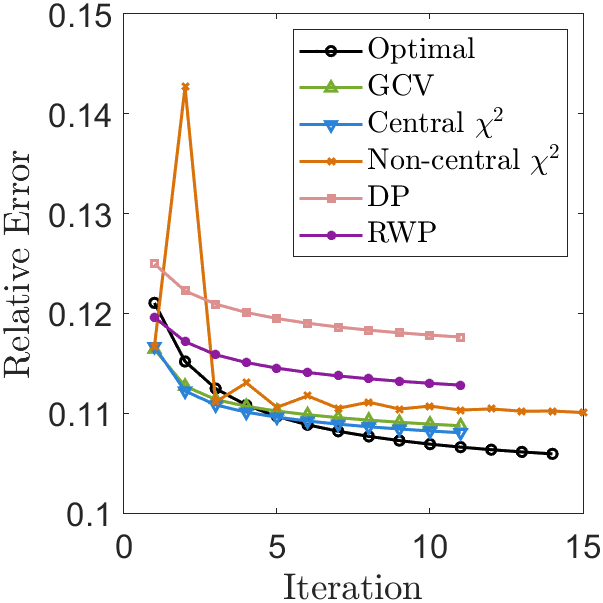}}
\subfigure[Relative change  in $\bfx$\label{fig:MM2D1:b}]{
\includegraphics[scale=.28]{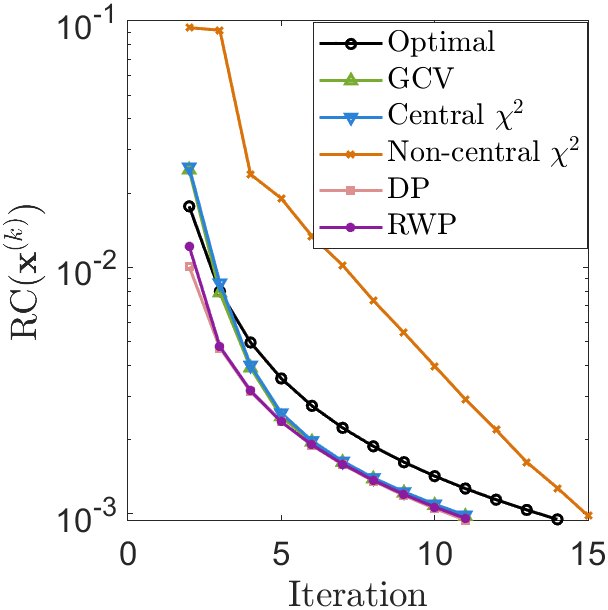}}
\subfigure[ Selection of $\lambda$ \label{fig:MM2D1:c}]{
\includegraphics[scale=.28]{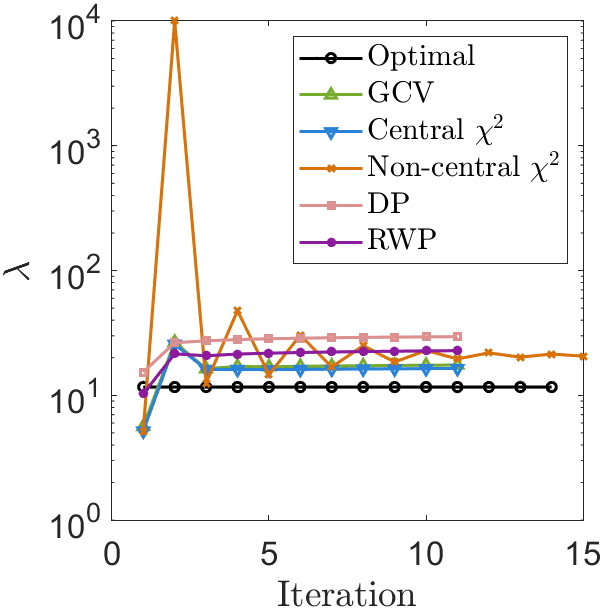}}
\subfigure[Relative change  in $\lambda^2$ \label{fig:MM2D1:d}]{
\includegraphics[scale=.28]{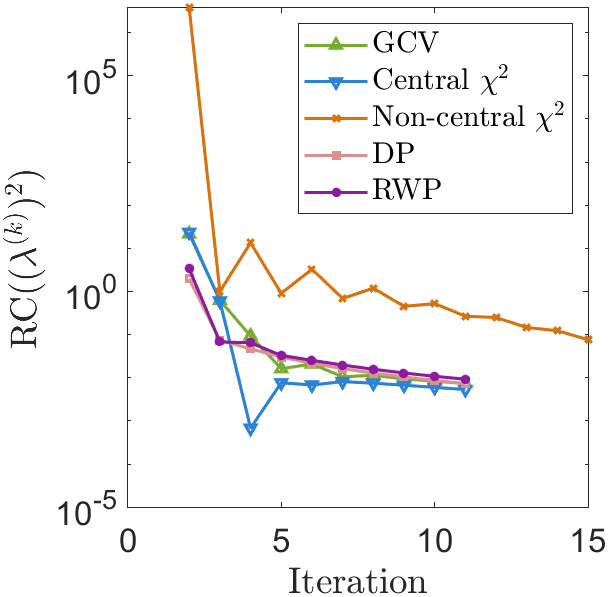}}
   \subfigure[ISNR \label{fig:MM2D1:e}]{
\includegraphics[scale=.28]{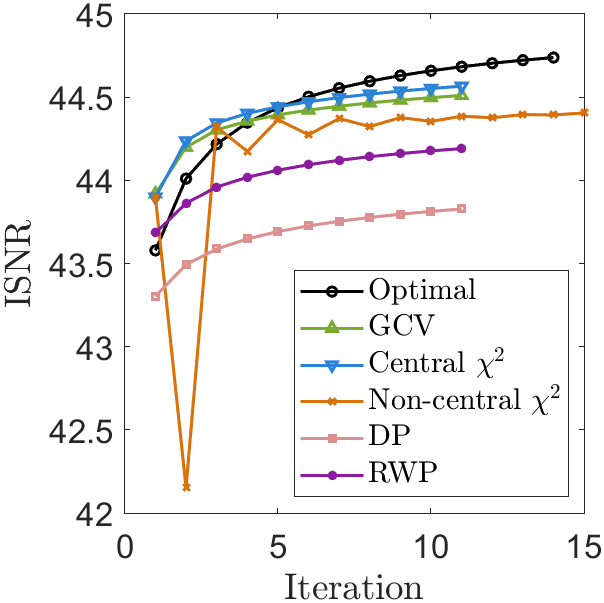}}
    \caption{
    Results for MM applied to \cref{fig:Ex2Setup:b} where $\lambda$ is fixed at the optimal $\lambda = 11.7$, or selected at each iteration with GCV, the $\chi^2$ dof tests, or DP.
    \Cref{fig:MM2D1:a} plots the $\text{RE}$ by iteration, \cref{fig:MM2D1:b} plots the relative change in $\bfx$, \cref{fig:MM2D1:c} plots the $\lambda$ selected, \cref{fig:MM2D1:d} plots the relative change in $\lambda^2$, and \cref{fig:MM2D1:e} plots the ISNR.}
    \label{fig:MM2D1}
\end{figure}

\begin{figure}
    \centering
            \subfigure[ $\bfx_{true}$ zoom-in \label{fig:Ex2Setupb:xt}]{
\includegraphics[width=0.18\textwidth]{Images/Cam512/PS22SB_XTinsert.png}}
\subfigure[Optimal $\lambda $ ($\lambda = 11.7$) 
\label{fig:MM2Res:a}]{
\includegraphics[width=0.18\textwidth]{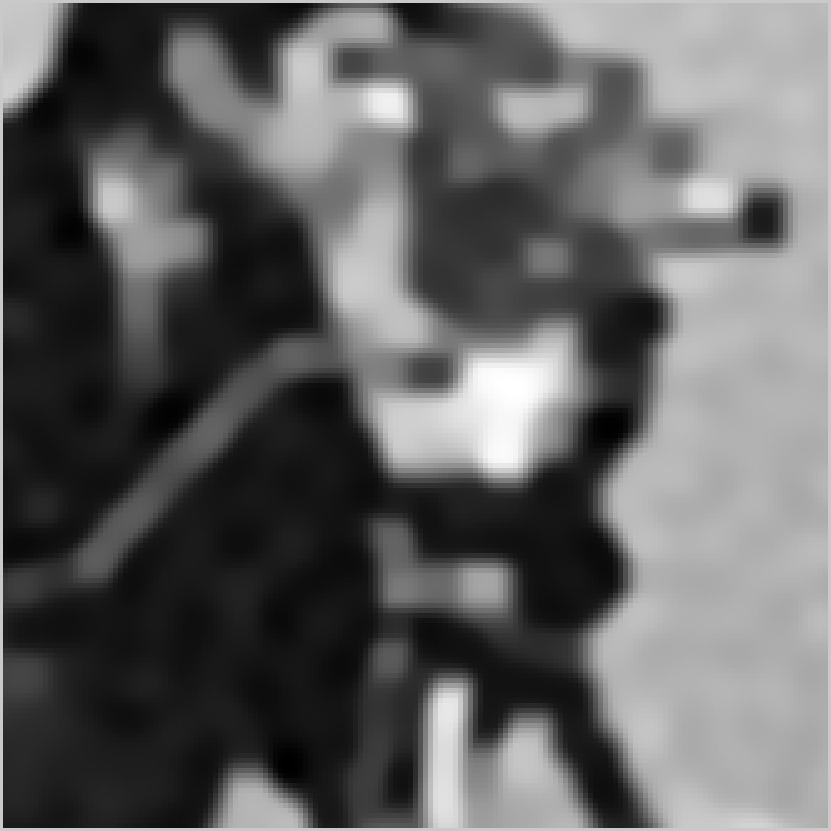}}

    \subfigure[ $\lambda$ selected by GCV
    \label{fig:MM2Res:b}]{
\includegraphics[width=0.18\textwidth]{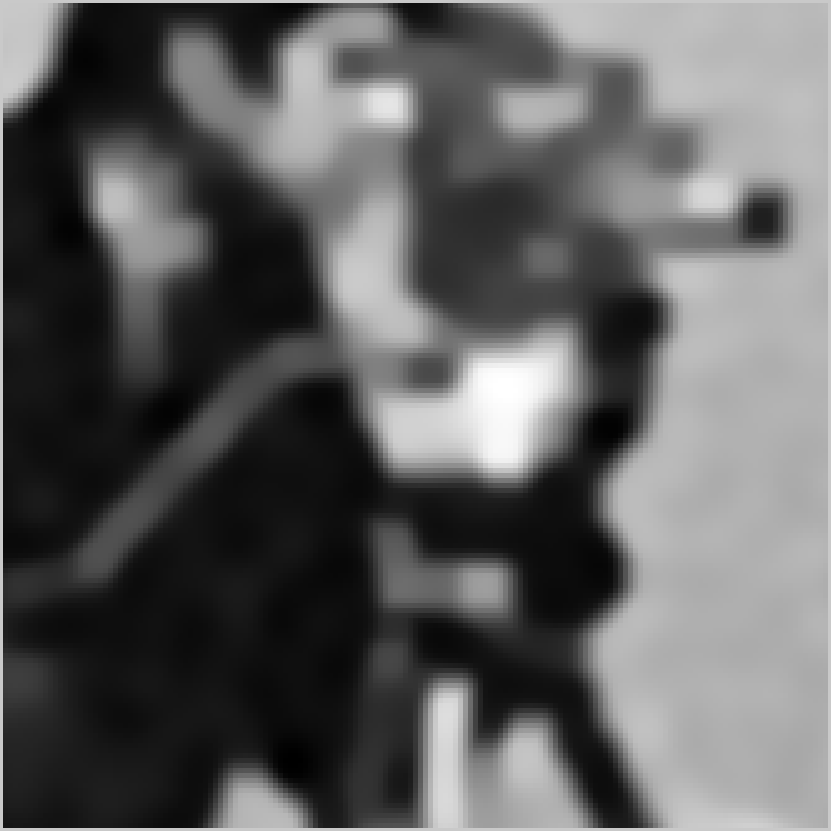}}
\subfigure[$\lambda$ selected by central $\chi^2$
\label{fig:MM2Res:c}]{
\includegraphics[width=0.18\textwidth]{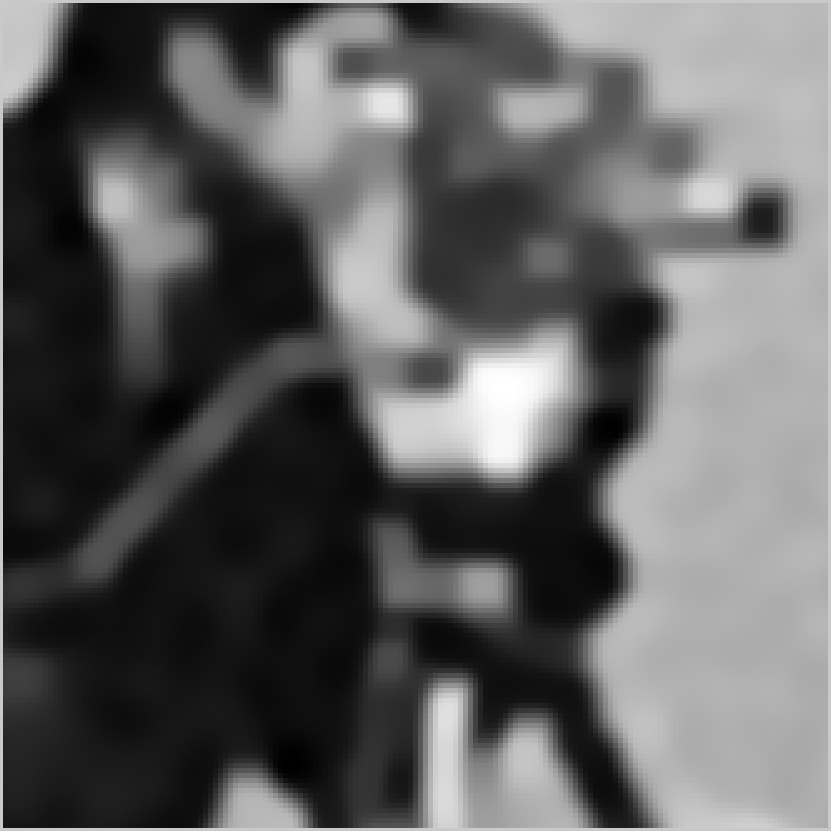}}
\subfigure[$\lambda$ selected by non-central $\chi^2$
\label{fig:MM2Res:d}]{
\includegraphics[width=0.18\textwidth]{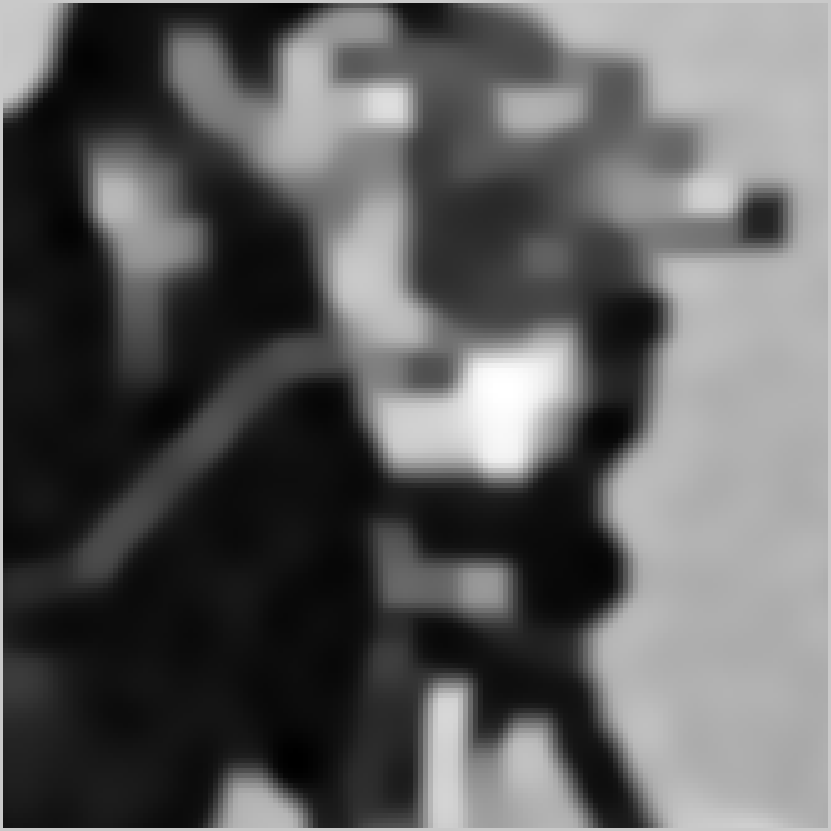}}
\subfigure[$\lambda$ selected by DP
\label{fig:MM2Res:e}]{
\includegraphics[width=0.18\textwidth]{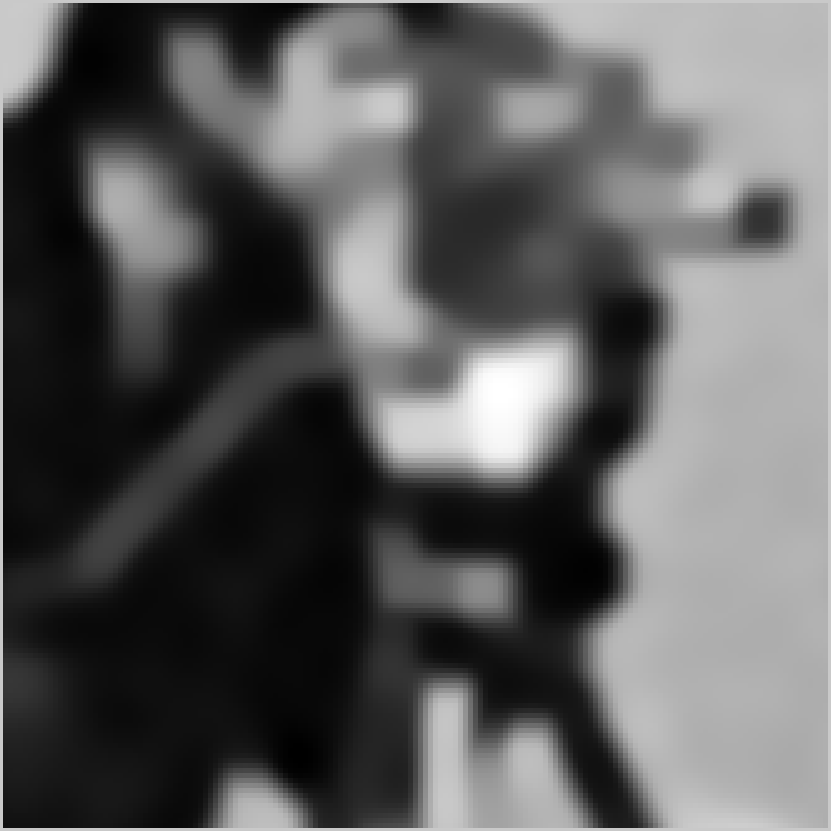}}
\subfigure[$\lambda$ selected by RWP
\label{fig:MM2Res:r}]{
\includegraphics[width=0.18\textwidth]{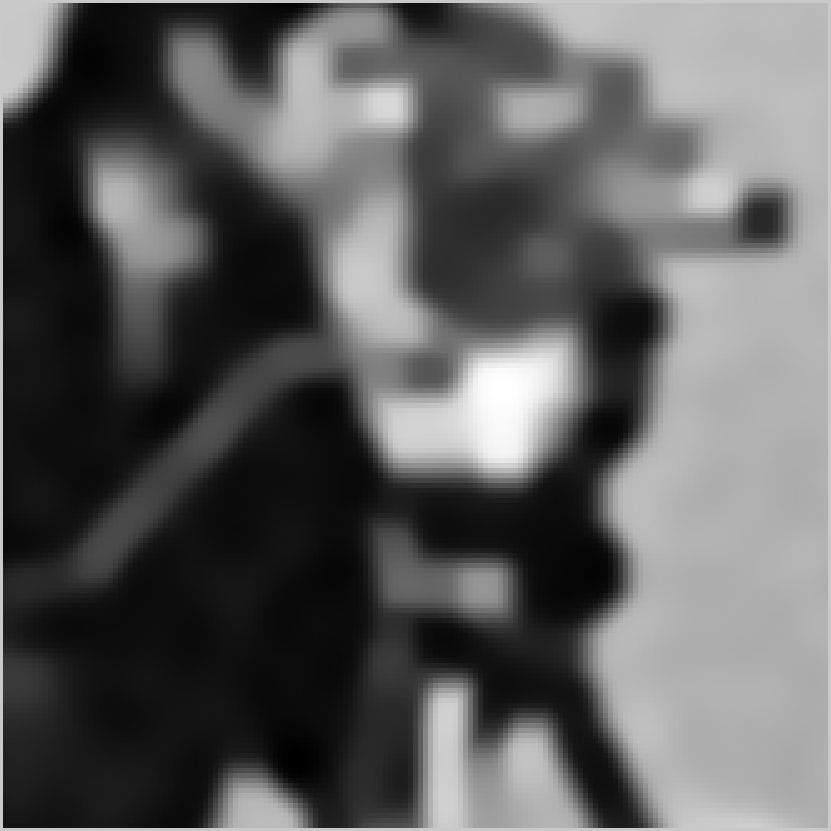}}
\subfigure[$\lambda$ selected by GCV, $\text{TOL}_\lambda=0.01$ \label{fig:MM2Res:f}]{
\includegraphics[width=0.18\textwidth]{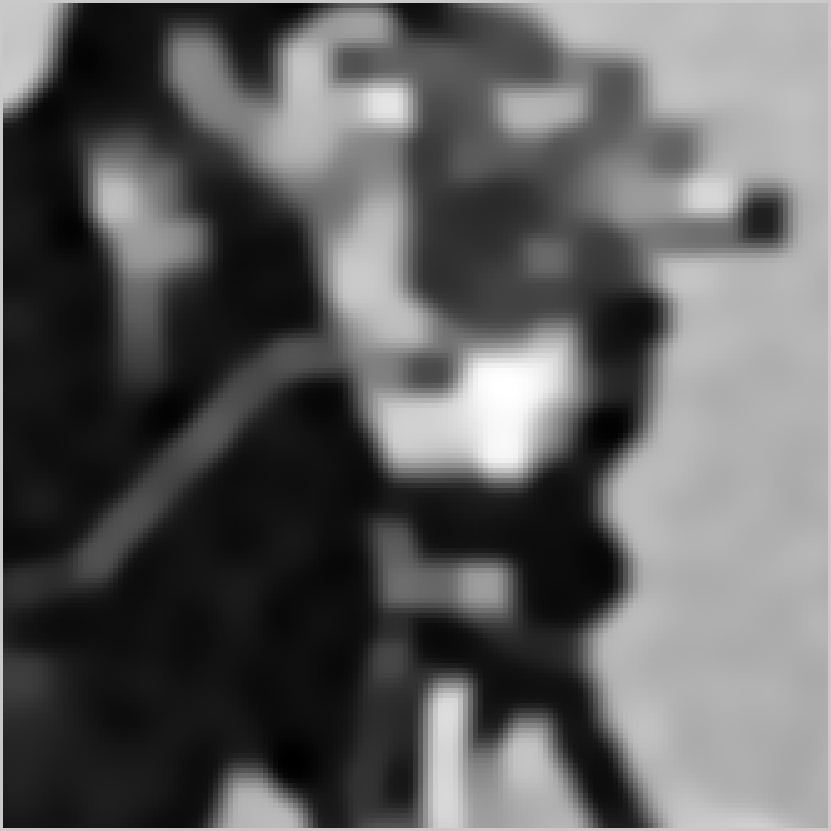}}
    \subfigure[ $\lambda$ selected by central $\chi^2$, $\text{TOL}_\lambda=0.01$
    \label{fig:MM2Res:g}]{
\includegraphics[width=0.18\textwidth]{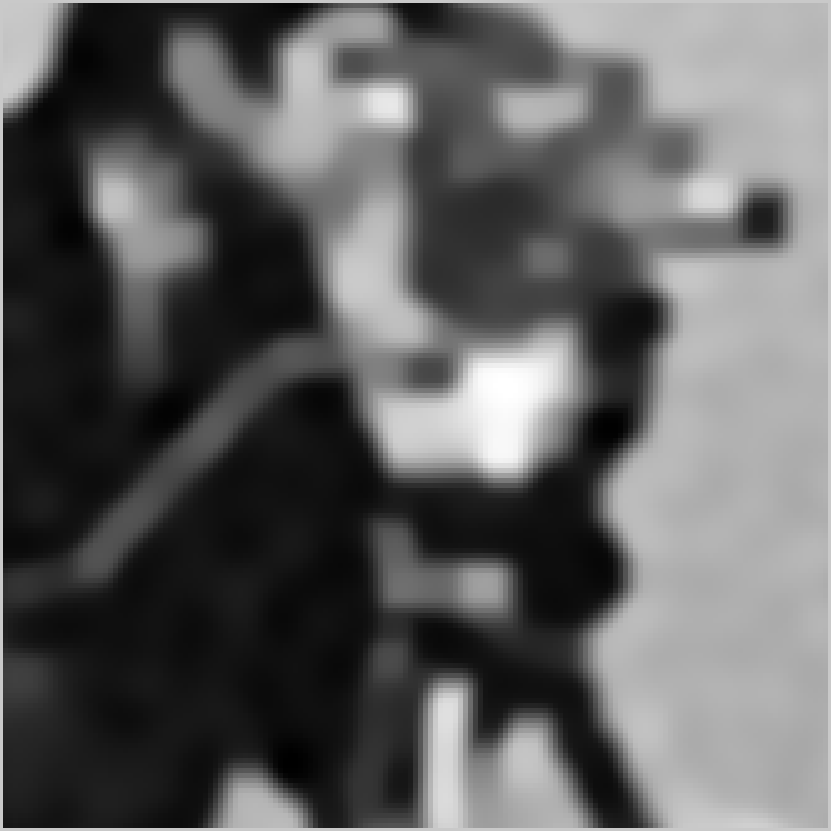}}
\subfigure[$\lambda$ selected by non-central $\chi^2$, $\text{TOL}_\lambda=0.01$ \label{fig:MM2Res:h}]{
\includegraphics[width=0.18\textwidth]{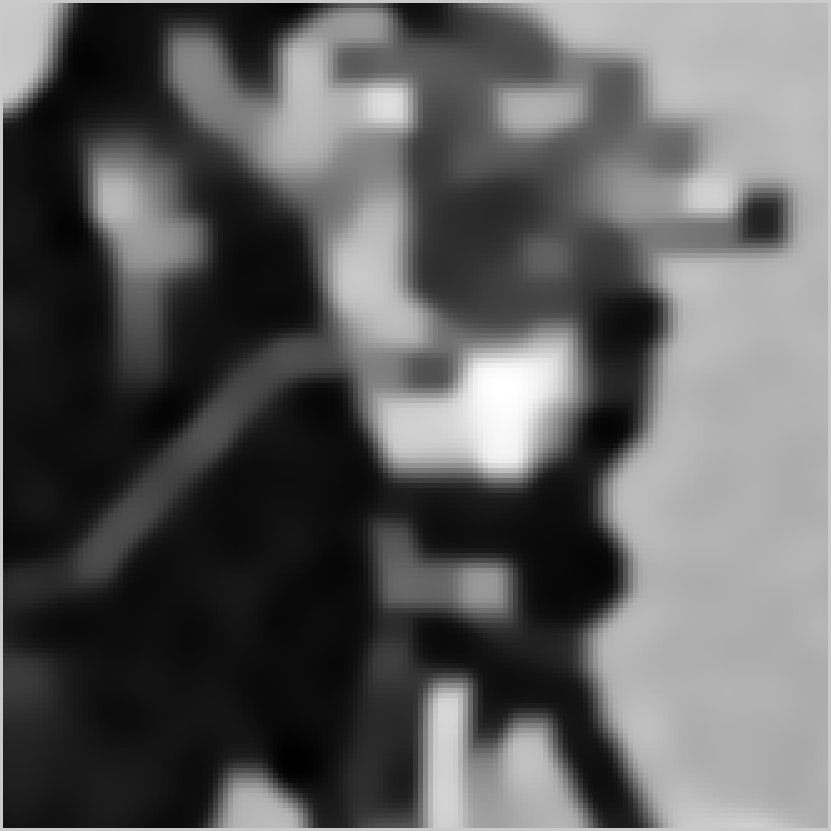}}
\subfigure[$\lambda$ selected by DP, $\text{TOL}_\lambda=0.01$
\label{fig:MM2Res:i}]{
\includegraphics[width=0.18\textwidth]{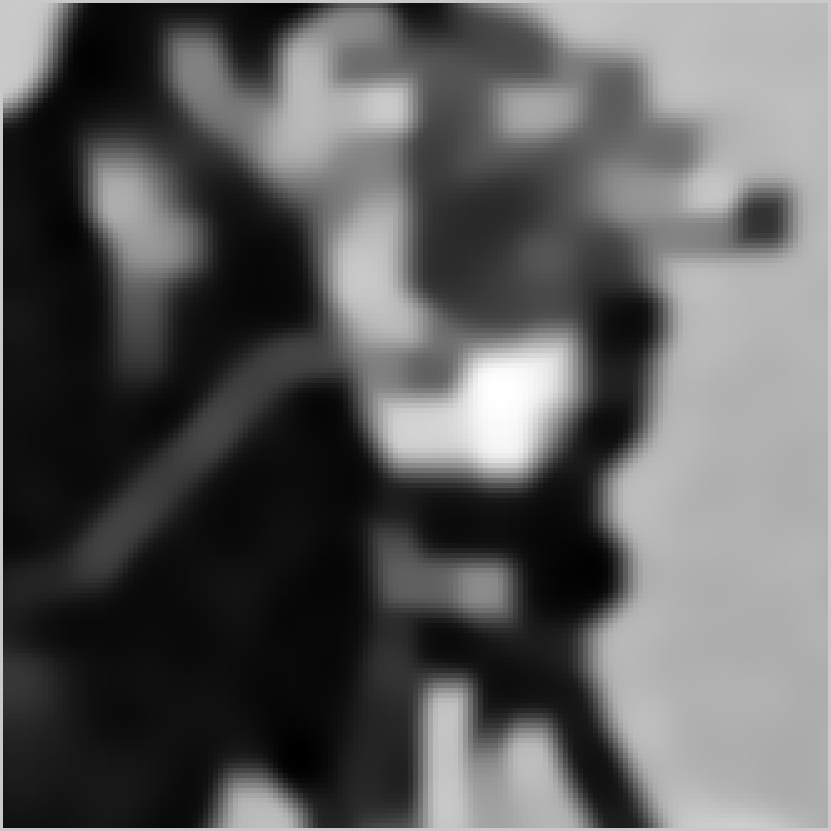}}
\subfigure[$\lambda$ selected by RWP, $\text{TOL}_\lambda=0.01$
\label{fig:MM2Res:rl}]{
\includegraphics[width=0.18\textwidth]{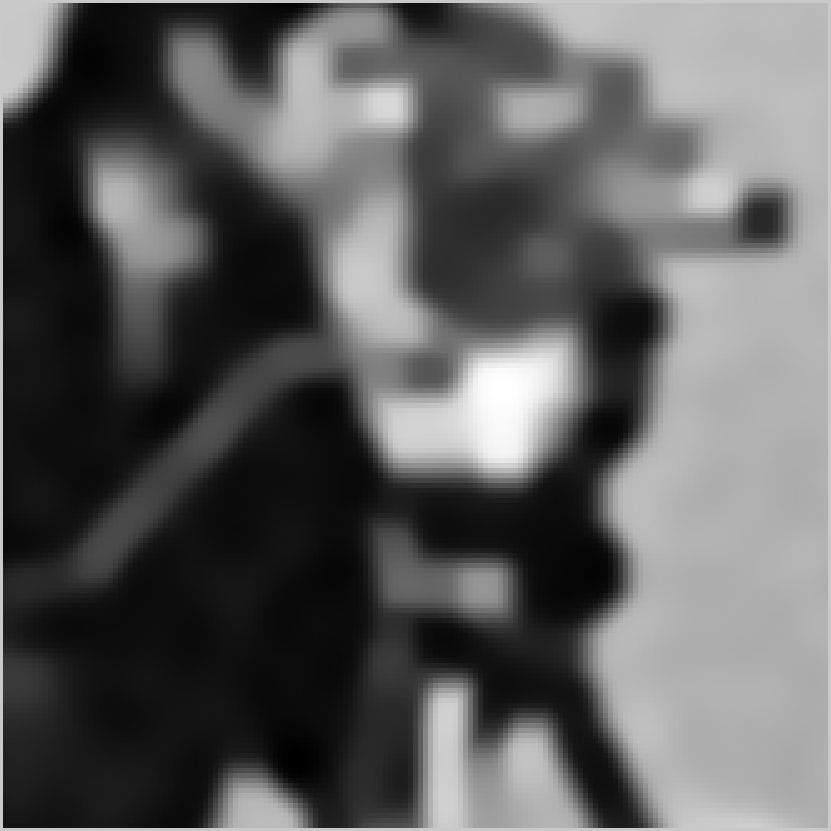}}
    \caption{Zoomed-in windows of the MM solutions at convergence for \cref{fig:Ex2Setup:b}. We note that the full images are indistinguishable and are thus not shown.}
    \label{fig:MM2Res}
\end{figure}

We set $\varepsilon =0.01$ as this is small relative to the magnitude of $\bfx$ \cite{buccini2021choice,buccini2019l}.  The results in \cref{fig:MM2D1} and the reconstructions in \cref{fig:MM2Res} show that for this $\varepsilon$, both GCV and the central $\chi^2$ dof test selection methods are comparable to fixing $\lambda$ at the \textit{optimal}, $\lambda = 11.7$. Again, the results are almost indistinguishable on the zoomed region with slightly greater smoothing evident for the RWP and DP methods.
The non-central $\chi^2$ dof test does not perform as well in this case, with the values of $\lambda$ oscillating widely in \cref{fig:MM2D1:c}.  DP and RWP again perform worse than GCV and the $\chi^2$ tests, with the solutions not being as clear near the handle and the camera stand.

\begin{table}[htp]
\caption{$\text{RE}$, $\text{ISNR}$, iterations, and computation time (in seconds) for the solutions to the 2D example in \cref{fig:Ex2Setup:b}. The best results are shown in boldface, excluding the methods where $\lambda$ is fixed at the \textit{optimal}.}
\centering
\begin{tabular}{lcccrcccr}
\hline
                         & \multicolumn{4}{c}{No $\text{TOL}_\lambda$}                                            & \multicolumn{4}{c}{$\text{TOL}_\lambda=0.01$}                                          \\ 
                         \cmidrule(lr){2-5} 
                         \cmidrule(lr){6-9}
Method                   & RE   & \multicolumn{1}{c}{ISNR} & \multicolumn{1}{c}{Iter.} & \multicolumn{1}{c}{Time} & RE   & \multicolumn{1}{c}{ISNR} & \multicolumn{1}{c}{Iter.} & \multicolumn{1}{c}{Time} \\ \hline
SB, Optimal              & 0.104 & 44.90                    & 19                        & 0.85                       &      &                          &                           &                          \\
SB, GCV                  & \textbf{0.104} & \textbf{44.91}                    & 17                        & 49.05                    & \textbf{0.104}                         &  \textbf{44.91} & 17                         &  40.04                        \\
SB, Central $\chi^2$     & \textbf{0.104} & \textbf{44.91}                    & 17                        & 54.43                    & \textbf{0.104} & \textbf{44.91}                    & 17                        & \textbf{21.91}                    \\
SB, Non-central $\chi^2$ & 0.106 & 44.72                    & 17                        & 135.54                   & 0.106 & 44.72                    & 17                        & 111.37                   \\
SB, DP                   & 0.113 & 44.15                    & 16                        & 78.64                    & 0.113 & 44.16                    & 17                        & 69.53                    \\
SB, RWP                   & 0.109 & 44.51                    & 17                        & 58.99                    & 0.109 & 44.51                    & 17                        & 58.99                    \\
\hline MM, Optimal              & 0.106 & 44.74                    & 14                        & 0.71                       &      &                          &                           &                          \\
MM, GCV                  & 0.109 & 44.51                    & \textbf{11}                        & 31.86                    &  0.109    & 44.51                         & \textbf{11}                          & 30.45                         \\
MM, Central $\chi^2$     & 0.108 & 44.57                    & \textbf{11}                        & 45.42                    & 0.108 & 44.57                    & \textbf{11}                        & 27.14                    \\
MM, Non-central $\chi^2$ & 0.110 & 44.41                    & 15                        & 145.17                   & 0.110 & 44.41                    & 15                        & 145.17                   \\
MM, DP                   & 0.118 & 43.83                    & \textbf{11}                        & 73.58                   & 0.118 & 43.83                    & \textbf{11}                        & 66.52                   \\ 
MM, RWP                   & 0.113 & 44.19                    & \textbf{11}                        & 38.11                   & 0.113 & 44.19                    & \textbf{11}                        & 38.11                   \\ 
\bottomrule
\end{tabular}
\label{tab:Ex2}
\end{table}

\subsubsection{Discussion on the two-dimensional results}\label{sec:twodres}
The results using the SB and MM iterative methods with the regularization parameter methods are summarized in \cref{tab:Ex2}.  In general, SB is better than MM in terms of RE and ISNR.  MM does take fewer iterations than SB, but the timing is close. We also observe that a tolerance should be set on $\lambda$ as this reduces the computational time while having little impact on the solution. The best method for the 2D problem is SB with central $\chi^2$ as this produces the best $\text{RE}$ and $\text{ISNR}$.  Both GCV and the central $\chi^2$ dof test perform well with both SB and MM, with solutions having $\text{RE}$ and $\text{ISNR}$ closest to the optimal. With $\text{TOL}_\lambda = 0.01$, central $\chi^2$ is faster than GCV. The central $\chi^2$ test, however, does require information on the noise while GCV does not. The non-central $\chi^2$ test and DP take longer to run and perform worse.  With a lower noise level, the $\chi^2$ test does not perform as well, so for a lower noise level, GCV is preferred while for a higher noise level, the central $\chi^2$ is better.

\section{Conclusions} \label{sec:Conclusion}

We have presented methods for selecting the parameters in the inner minimization problems of SB and MM by using GCV or the $\chi^2$ dof test at each iteration, including showing a new approach to provide an estimate of the expected value of $\bfx$ each iteration, and a new theorem on the $\chi^2$ degrees of freedom when $p>n$. For the non-central $\chi^2$ dof test, we proposed using the current solution in the iterative method as the mean of the solution.  Although the parameters selected in this method vary in the early iterations, they still converge once $\bfx^{(k)}$ is closer to convergence. Numerical examples demonstrate that selecting the parameter at each iteration with these methods produces comparable results in terms of the final relative error and the number of iterations to using the \textit{optimal} fixed parameter.  
In addition, these methods do not need to be used at every iteration and can still be helpful for finding a suitable parameter in the initial iterations. They zoom in on the ideal parameter which can then be fixed after the selection method converges. This still performs well and is computationally cheaper than searching for the fixed parameters by running SB or MM to completion multiple times. 

\section*{Acknowledgments}
Funding: This work was partially supported by the National Science Foundation (NSF) under grant 
DMS-1913136, and DMS-2152704 for  Renaut.  Any opinions, findings, conclusions, or recommendations expressed in this material are those of the authors and do not necessarily reflect the views of the National Science Foundation. M.I. Espa\~nol was supported through a Karen Uhlenbeck EDGE Fellowship.

\bibliographystyle{siamplain}

\bibliography{references}

\end{document}